\documentclass[hidelinks,onefignum,onetabnum]{siamart220329}

\usepackage[algo2e,boxed]{algorithm2e}
\usepackage{amsfonts}
\usepackage{bm} 
\usepackage{enumerate}
\usepackage{graphicx}
\usepackage{mathrsfs}
\usepackage{pstricks}
\usepackage{subfigure}
\usepackage{tikz-cd}
\usepackage{url}
\usepackage{algorithm}
\usepackage{algorithmic} 

\newtheorem{ntn}[theorem]{Notation}
\newtheorem{cvn}[theorem]{Convention}

\makeatletter

\usepackage{lipsum}
\usepackage{amsmath}
\usepackage{amsfonts}
\usepackage{graphicx}
\usepackage{epstopdf}
\usepackage{enumerate}
\usepackage{algorithmic}
\usepackage{subfigure}
\usepackage{algorithm}
\usepackage{algorithmic} 
\usepackage{booktabs}
\usepackage{array}
\Crefname{ALC@unique}{Line}{Lines} 

\ifpdf
  \DeclareGraphicsExtensions{.eps,.pdf,.png,.jpg}
\else
  \DeclareGraphicsExtensions{.eps}
\fi

\newcommand{\Dim}{{ \textup{D}}}
\newcommand{\Int}{\mathrm{int}}

\newsiamremark{remark}{Remark}
\newsiamremark{hypothesis}{Hypothesis}
\crefname{hypothesis}{Hypothesis}{Hypotheses}
\newsiamthm{claim}{Claim}
\newsiamthm{assumption}{Assumption}

\usepackage{amsopn}

\headers{Topology of 3D Continua with Singular Points}
{H. Liang, Y. Qiu, Y. Tan, and Q. Zhang}

\title{
  On Topology of Three-dimensional Continua \\
  with Singular Points\thanks{
    Qinghai Zhang is the corresponding author.
    Author names are listed in the alphabetical ordering. 
    \funding{This work was supported by
      the National Natural Science Foundation of China (\#12272346)
      and the Fundamental Research Funds
      for the Central Universities 226-2025-00254.}}
}

\author{
  Hao Liang\thanks{School of Mathematics, Foshan University, Foshan, Guangdong,
    528000, China.(\email{lianghao1019@-\\hotmail.com}).}
  \and
  Yunhao Qiu\thanks{School of Mathematical Sciences,
    Zhejiang University, Hangzhou, Zhejiang, 310058, China.
    (\email{qiuyunhao@zju.edu.cn},
    \email{yantanhn@zju.edu.cn},
    \email{qinghai@zju.edu.cn})    
  }
  \and
  Yan Tan\footnotemark[3] 
  \and
  Qinghai Zhang\footnotemark[3] 
  \thanks{
    Institute of Fundamental and Transdisciplinary Research,
    Zhejiang University, Hangzhou, Zhejiang, 310058,
    China.
}
}

\begin{document}

\maketitle

\begin{abstract}
  We propose to model
 the topology of three-dimensional (3D) continua by Yin sets,
 regular open semianalytic sets with bounded boundary.
Our model differs from manifold-based models
 in that singular points of a 3D continuum, i.e.,
 boundary points where the tangent plane is not uniquely defined, 
 are treated not as anomalies
 but as a central subject of our theoretical investigation.
We characterize the local and global topology of Yin sets.
Then we give a unique boundary representation (B-rep) of Yin sets 
 based on the notion of a glued surface, 
 a quotient space of an orientable compact 2-manifold 
 along a one-dimensional CW complex. 
As a prominent advantage of this B-rep,
 global topological invarants such as the number of connected
 components and the number of holes (2-cycles) in each component
 can be extracted in $O(1)$ time. 
Our results apply to 3D continua with arbitrarily complex topology
 and may be useful
 in a number of scientific, engineering, and industrial applications
 such as solid modeling, computer-aided design, 
 and numerical simulations of multiphase flows with topological changes.

\end{abstract}



%


\section{Introduction}
\label{sec:intro}

The concept of 
 a continuum is of fundamental significance
 in a vast number of scientific and engineering disciplines,  
 such as continuum mechanics, solid modeling, 
 and computer-aided design (CAD). 
Thus it is desirable 
 to model continua both theoretically and algorithmically
 so as to describe their topology, approximate their geometry, 
 and answer queries on derived quantities
 such as the normal vector, the interface curvature,
 and even the Betti numbers.

\subsection{Motivations from solid modeling and CAD}
\label{sec:motiv-from-solid}

There exist an array of continua models with 
 varying mathematical rigor and application ranges. 
In the theoretical model of Requicha and Voelcker
 \cite{requicha83:_solid}, 
 continua in three dimensions (3D) are represented
 by r-sets, i.e., bounded regular closed semianalytic sets
 defined in \Cref{sec:YinSets}; 
 r-sets have been an essential mathematical tool
 for solid modeling since the dawning time of this field
 \cite{requicha78:_mathem_found_const_solid_geomet}.
Unfortunately, as illustrated by \cite[Fig. 5]{zhang20:_boolean}, 
 the boundary representation (B-rep) of r-sets
 is not unique.
For the uniqueness of the B-rep
 and the convenience of topological classification
 of 3D continua via their boundaries, 
 it is desirable to have a new theoretical model
 where 3D continua and their boundaries
 have a one-to-one correspondence.

Algorithmically,
 the boundary 
 of a 3D continuum ${\cal Y}$ is often assumed 
 to be a 2-manifold.
Although this assumption leads to
 an efficient B-rep of ${\cal Y}$, 
 it does not cover all legitimate cases of 3D continua. 
As shown in \Cref{fig:yinsetexm}, 
 it is possible that
 a boundary point 
 is a non-manifold point, 
 cf. \Cref{def:manifoldPoints}, 
 a particular type of singular points
 in \Cref{def:singularPoints}.
These non-manifold points may cause,
 among other things,
 difficulties in the B-rep of 3D continua. 
In most CAD algorithms,
 the B-rep surface of a 3D continuum
 is either assumed to be free of non-manifold/singular points 
 or changed to a 2-manifold
 by removing such points
 \cite{geuziec01:_cutting_and_stitching}.

Roughly speaking,
 the mesh data of current 3D continua models 
 come either from digitizing real-world phenomena/objects
 or from synthesizing virtual data
 such as those created by human designers
 or computer programs such as a grid generator.
Examples of digitized models include those
 in medical imaging \cite{zhang02:_direct_surface_extract},
 surface scanning \cite{bernardini02:_3d_model_aqui_pipe},
 and 3D photography
 \cite{seitz06:_compar_evalua_multi_view_stereo_recons}; 
 data in this category usually
 contain the so-called \emph{defects} or \emph{flaws}
 such as holes and self-intersections. 
Examples of synthetic models include 
 sketch-based models \cite{igarashi06:_teddy_sketch_interface_free_design}, 
 models with implicit mathematical formulations \cite{turk02:_model_with_implicit_surf},
 and CAD systems \cite{farouki99:_close_gap_betw_cad_model}; 
 data in this category 
 are typically assembled from multiple parts
 and may contain other types of flaws such as
 overlaps, pairwise intersections, and self-intersections.
Before these data can be used, 
 their flaws have to be removed
 by a manual and tedious preprocessing known as \emph{mesh repairing}, 
 which has been under intense research; 
 see \cite{attene:13_polygon_mesh_repairing}
 for a survey on how to select mesh repairers
 for given downstream applications.
Things are made even more complicated by the fact that 
 the precise definition of defects or flaws may vary
 according to the downstream applications.
To sum up,
 singular points of 3D continua
 are theoretically treated as anomalies in CAD models 
 and algorithmically handled 
 on a case-by-case basis, 
 leading to algorithmic complexity,
 code bloating, and high maintenance of the software.
 
From a physical viewpoint,
 some defects are \emph{legitimate}
 in that it is possible for 3D objects or real phenomena to display them
 while others are \emph{illegal}
 in that they cannot possibly exist
 in any physically meaningful 3D continua. 
For the examples shown in \mbox{\cite[Fig. 1]{attene:13_polygon_mesh_repairing}}, 
 a singular vertex is legitimate 
 while any hole, i.e.,
 the missing of a triangle (a 2-simplex)
 in a closed surface (a 2-simplicial complex), is illegal. 
Indeed, as indicated by \Cref{fig:yinsetexm}(a),
 a singular vertex or a curve of such points may show up
 during merging or pinch-off. 
In contrast, any hole is illegal
 because it incurs an inconsistency
 in the B-rep of a 3D continuum. 
Therefore,
 for current algorithms in solid modeling and CAD, 
 it might be useful to enumerate and classify all possible scenarios
 of singular points of 3D continua.
Such a result is given in \Cref{thm:globalTopology}.

\subsection{Motivations from integrating CAD and CAE (Computer-Aided Engineering)}
\label{sec:motiv-from-integr}

During the development of a new product,
 CAD systems help the developer with the geometric design of the prototype
 while CAE systems test its performance by computer simulations.
For example, 
 the shape and assembly of a water turbine 
 are designed inside a CAD system
 such as \mbox{AutoCAD\texttrademark} \
 while the performance of the turbine
 is tested inside a CAE system such as Fluent\texttrademark.
The feedback provided by the CAE system may lead to 
 a change of the geometric design, 
 which, in turn, leads to new CAE results; 
 this process is usually repeated
 until both the shape and the performance of the updated prototype
 become acceptable. 

Traditionally, CAD designers and CAE specialists
 are two disparate groups of people,
 whose communications are very limited and infrequent.
Their distinct sets of skills and different viewpoints
 often lead to divergent opinions.
Consequently,
 the iteration between CAD and CAE tends to be very slow
 and it often takes a long time to reach a consensus of the two parts. 

With the fast advance of science, engineering, and industry, 
 there has been a need to integrate CAD and CAE
 to speed up the pipeline of designing new products;
 see, e.g.,  \cite{tautges00:_common_geomet_modul_cgm,gujarathi11:_para_cad_cae_inter_using_common_data_model,khan21:_cad_cae_hdf}
 and references therein.
In most of these efforts, 
 CAD and CAE software are treated as black boxes 
 and a middleware is developed to exchange data between them. 
One major issue of this approach is the almost unavoidable information loss
 between CAD and CAE,
 in which the two sets of data have, apart from the different formats,
 more fundamental discrepancies.
For example, the inconsistent definitions of tolerances
 make it difficult
 to rebuild the correct topology of CAD models for CAE.
It is also a formidable task to bind different levels of geometric parts
 into integrity that is recognizable by CAE algorithms
 \cite{cao09:_cad_cae_inte_frame_layer_soft_archi}
 \cite[Fig. 1]{khan21:_cad_cae_hdf}. 

Based on the classification of \emph{all} 3D continua in \Cref{thm:globalTopology},
 we propose to separate the description of their topology
 from the approximation of their geometry.
One obvious advantage of this B-rep approach
 is the ease of topology rebuilding and geometric integration
 during the process of data share and exchange:
 the topological structure is agnostic of geometric tolerances
 and different requirements of the geometric approximation
 do no affect the topology.
When the shapes of a CAD model are specified via
 a set of rational parametric surfaces,
 one can use the outstanding algorithm 
 by Jia, Chen, and Yao \cite{jia22:_singularity_comput} 
 to compute all singular points
 and combine \Cref{thm:globalTopology}
 with the unique B-rep of 3D continua in \Cref{thm:connectedYinsetrep}
 to integrate CAD and CAE
 with reduced information losses. 
Also, due to the almost disjoint surfaces in
 \Cref{def:almostDisjointGsurf}, 
 it is sometimes unnecessary for this integration 
 to compute self-intersections
 \cite{li25:_fast_determ_comput_self_intersect_nurbs}
 and overlap regions \cite{yang25:_overlap_region_extra_two_nurbs}
 of B-rep surfaces. 

\subsection{Motivations from interface tracking (IT)
   in simulating multiphase flows}
\label{sec:interf-track-it}

In computational physics,
 recent advances of numerically solving partial differential equations
 on moving domains
 call for high-fidelity simulations 
 that integrate CAD models of complex topology/geometry
 with CAE algorithms/analysis.
One such field is the study of multiphase flows.
 
In drastic comparison to solid modeling,
 a rigorous model of the topology and geometry of 3D continua 
 appears to be absent in continuum mechanics; 
 this is particularly true in computational fluid dynamics
 and multiphase flows. 
In a sharp-interface model of multiphase flows, 
 each phase is represented by a moving continuum, 
 which evolves according to the governing equations
 and the interface/boundary conditions.
Each phase may undergo large geometric deformations
 and even topological changes.
The accurate IT of these phases 
 is crucial in numerically simulating multiphase flows,
 because the IT error not only leads to a lower bound
 on estimation errors of 
 the normal vector and interface curvature \cite{zhang17:_hfes},
 but also adversely affect the accuracy of
 simulating the flow field.

Currently, popular IT methods include
 the volume-of-fluid (VOF) method \cite{Hirt.Nichols_1981_volume},
 the level-set method \cite{osher88:_front_propag_curvat_speed},
 and the front tracking method
 \cite{tryggvason01:_front_track_method_comput_multip_flow}. 
In VOF methods,
 the region occupied by a fluid phase $\mathsf{M}$
 is modeled by the point set
 $\mathcal{M}(t):=\{\mathbf{x}: \chi(\mathbf{x},t)=1\}$
 where the \emph{color function} $\chi(\mathbf{x},t)$
 is 1 if there is $\mathsf{M}$ at $({\bf x},t)$
 and 0 otherwise.
At each time $t$,
 $\mathcal{M}(t)$ is implicitly represented
 by volume fractions of $\chi$
 in all control volumes.
Within a given time interval, 
 ${\cal M}(t)$ is evolved by numerically solving
 the governing equation of $\chi$,
 be it the scalar conservation law
 or the advection equation. 
%
In level-set methods,
 an interface is implicitly represented
 as the zero isocontour of a signed distance function,
 and is tracked by numerically solving
 the scalar conservation law or the advection equation. 
In the front tracking method
 \cite{tryggvason01:_front_track_method_comput_multip_flow},
 the boundary of a phase
 is explicitly represented by a set of connected Lagrangian markers
 and the IT of this phase is reduced to tracking these markers
 via numerically solving ordinary differential equations.
In all of the aforementioned IT methods,
 \emph{topological and geometric problems in IT are avoided
 as they are converted to
 numerically solving differential equations}.

Our approach to IT is fundamentally different:
 \emph{we tackle topological and geometric problems in IT
 with tools in topology and geometry}.
This principle of \emph{fluid modeling}
 has been directing our work in two dimensions (2D): 
 we have established a model of 2D continua  
 called Yin sets \cite{zhang20:_boolean},
 completely classified all 2D Yin sets,
 designed a unique B-rep of 2D continua, 
 equipped the Yin sets with a natural Boolean algebra,
 proposed an analytic framework called MARS
 (mapping and adjusting regular semianalytic sets)
 for analyzing explicit IT methods \cite{zhang16:_mars},
 and developed several MARS methods
 \cite{zhang18:_cubic_mars_method_fourt_sixth,TaQi25}
 for high-order IT in 2D.
By performing standard benchmark tests
 for both two-phase IT \cite{zhang18:_cubic_mars_method_fourt_sixth}
 and three-or-more-phases IT \cite{TaQi25}, 
 we have shown that the MARS methods
 are superior over current IT methods
 in terms of accuracy, efficiency,
 and preserving topological structures and geometric features. 
Given the explicit modeling of 2D fluids, 
 the supremacy of MARS methods is not surprising.

A major motivation of this work
 is the generalization of fluid modeling from 2D to 3D.
As a fundamental difference of IT from CAD,
 the singular points
 that characterize topological changes of a continuum
 should never be removed.
Instead,
 these singular points should be in the spotlight of both theory and algorithms
 so that homeomorphic deformations
 and topological changes of a 3D continuum
 can be handled accurately and efficiently
 for arbitrarily complex topology. 
 
\subsection{Contributions of this work}
\label{sec:contr-this-work}

The discussions in \Cref{sec:motiv-from-solid,sec:motiv-from-integr,sec:interf-track-it}
 manifest common topological and geometric problems
 in a wide range of scientific, engineering, and industrial endeavors.
 
\begin{enumerate}[(Q-1)]
 \item Can we design, 
   from a minimum set of natural physical constraints,
   a 3D continua model 
   that neither includes nonphysical singular points 
   nor excludes the legitimate ones?
 \item Can we rigorously characterize the local and global topology
   of 3D continua with singular points?
 \item Based on the answer to (Q-2),
   can we design a scheme of unique B-rep of 3D continua
   by a set of surfaces that are free of overlaps,
   self-intersections, and proper pairwise intersections? 
 \item Sometimes the physics of a 3D continuum
   depends on global topological invariants
   such as the number of connected components of a fluid phase
   and the number of holes inside a connected component.
   For example,
   the number of bubbles per unit volume in a bubbly flow
   is a crucial characterization of the fluid properties.
   So can we extract these topological invariants in $O(1)$ time?
 \end{enumerate}

In this work, we give positive answers to all the above questions.
 
In \Cref{sec:YinSets},
 we briefly review the Yin space proposed in \cite{zhang20:_boolean};
 it is fortuitous that the Yin sets in \Cref{def:YinSet}
 originally designed for 2D continua 
 worked out in 3D without any modifications. 
The 3D Yin space is also our answer to (Q-1). 
In \Cref{sec:singularPoints},
 all characteristic boundary points of a 3D Yin set
 are included in the notion of singular points.
In \Cref{sec:compact-2-manifolds},
 we define a surface as an orientable compact 2-manifold
 and introduce the important notions
 of proper and improper intersections of surfaces.
In \Cref{sec:localTopology}, 
 we examine the local topology of 3D continua 
 by focusing on the good neighborhood of singular points.
The key concept of good pairing in \Cref{sec:good-pairing}
 leads to \Cref{thm:globalTopology} on the global topology, 
 which answers (Q-2).
Different from the 2D case, 
 it is not obvious
 how to deduce a unique B-rep of 3D continua
 from the main theorem on the global topology.
In \Cref{sec:representation}, 
 we define the new concept of glued surfaces, 
 employ it as the basic building block
 in decomposing the boundary of a 3D continuum,
 and show in \Cref{thm:connectedYinsetrep}
 that any connected 3D Yin set
 can be uniquely represented
 by a set of oriented glued surfaces.
Apart from answering (Q-3,4),
 this unique 3D B-rep
 has the same form as that of the 2D case,
 with glued surfaces replacing the Jordan curves
 in \cite[eq. (3.6)]{zhang20:_boolean}.

Following the theory \cite{zhang20:_boolean}
 and algorithms \cite{TaQi25} of 2D fluid modeling, 
 this work is our first theoretical step in 3D fluid modeling
 for multiphase flows. 




\section{The Yin space}
\label{sec:YinSets}

In a topological space ${\mathcal X}$,
 the \emph{complement} of a subset ${\mathcal P}\subseteq {\mathcal X}$,
 written ${\mathcal P}'$,
 is the set ${\mathcal X}\setminus {\mathcal P}$.
The \emph{closure} of a set ${\mathcal P}\subseteq{\mathcal X}$,
 written $\overline{{\mathcal P}}$,
 is the intersection of all closed 
 supersets of ${\mathcal P}$.
The \emph{interior} of ${\mathcal P}$, written ${\mathcal P}^{\circ}$,
 is the union of all open subsets of ${\mathcal P}$.
The \emph{exterior} of ${\mathcal P}$,
 written ${\mathcal P}^{\perp}:= {\mathcal P}^{\prime\circ}
 :=({\mathcal P}')^{\circ}$,
 is the interior of its complement.
By the identity $\overline{{\mathcal P}} ={\mathcal P}^{\prime\circ\prime}$
 \cite[p. 58]{givant09:_introd_boolean_algeb},
 we have 
 ${\mathcal P}^{\perp}={\overline{\mathcal P}}^{\prime}$.
A point $\mathbf{x}\in {\mathcal X}$ is
 a \emph{boundary point} of ${\mathcal P}$
 if $\mathbf{x}\not\in {\mathcal P}^{\circ}$
 and $\mathbf{x}\not\in {\mathcal P}^{\perp}$.
The \emph{boundary} of ${\mathcal P}$, written $\partial {\mathcal P}$,
 is the set of all boundary points of ${\mathcal P}$.
It can be shown that
 ${\mathcal P}^{\circ}={\mathcal P}\setminus \partial {\mathcal P}$
 and
 ${\overline{\mathcal P}}= {\mathcal P}\cup \partial {\mathcal P}$.
A set ${\mathcal P}\subseteq {\mathcal X}$ is \emph{regular open} if
 it coincides with the interior of its own closure,
 i.e. if ${\mathcal P}={\overline{\mathcal P}}^{\circ}$,
 and is \emph{regular closed} if
 it coincides with the closure of its own interior,
 i.e. if ${\mathcal P}=\overline{{\mathcal P}^{\circ}}$.
The duality of the interior and closure operators
 implies
 ${\mathcal P}^{\circ} =\overline{{\mathcal P}^{\prime}}^{\prime}$,
 hence ${\mathcal P}$ is a regular open set
 if and only if ${\mathcal P}={\mathcal P}^{\perp\perp}
 :=\left({\mathcal P}^{\perp}\right)^{\perp}$.
For any subset $Q\subseteq{\mathcal X}$,
 it can be shown that $Q^{\perp\perp}$ is a regular open set
 and $\overline{Q^{\circ}}$ is a regular closed set.

Regular sets, open or closed, 
 capture the salient feature 
 that continua are free of lower-dimensional elements
 such as isolated points and curves in $\mathbb{R}^2$
 and dangling faces in $\mathbb{R}^3$.
However, regular sets are not perfect for representing
 physically meaningful regions yet: 
 some of them
 cannot be described by a finite number of symbol structures.
For example, some sets
 have nowhere differentiable boundaries,
 which, in their parametric forms,
 are usually infinite series of continuous functions
 \cite{sagan92:_elemen_proof_schoen_space_fillin}.
Another pathological case
 is more subtle:
 intersecting two regular sets 
 may yield an infinite number of disjoint regular sets.
Consider 
 \begin{equation}
   \label{eq:pathologicalIntersection}
   \left\{
   \begin{array}{l}
   {\mathcal A}_p := \{(x,y)\in \mathbb{R}^2 :
    -2< y< \sin\frac{1}{x},\ 
    0< x< 1 \},\\
   {\mathcal A}_s := \{(x,y)\in \mathbb{R}^2 :
    0 < y < 1,\ 
    -1< x< 1 \}.
   \end{array}
   \right.
 \end{equation}

Although both ${\mathcal A}_p$ and ${\mathcal A}_s$
 are described by two inequalities,
 their intersection
 is a disjoint union of an infinite number of regular sets;
 see
 \cite[Fig. 4-1, Fig. 4-2]{requicha77:_mathem_model_rigid_solid_objec}.
This poses a fundamental problem
 that results of Boolean operations of two regular sets
 may not be well represented on a computer
 by a finite number of entitites.

Therefore, 
 we need to find a proper subspace of regular sets,
 each element of which is finitely describable.
This search eventually
 arrives at semianalytic sets.

\begin{definition}[Semianalytic sets]
  \label{def:semianalytic-sets}
  A set ${\mathcal S}\subseteq\mathbb{R}^{\Dim}$ is \emph{semianalytic}
   if there exist a finite number of 
   analytic functions $g_i:\mathbb{R}^{\Dim}\rightarrow \mathbb{R}$
   such that ${\mathcal S}$ is in the universe
   of a finite Boolean algebra formed from the sets
   \begin{equation}
     \label{eq:semiAnalyticForm}
     {\mathcal X}_i=\left\{\mathbf{x}\in \mathbb{R}^{\Dim}
       :  g_i(\mathbf{x})\ge 0\right\}.
   \end{equation}
  The $g_i$'s are called the \emph{generating functions}
   of ${\mathcal S}$.
  In particular, a semianalytic set is \emph{semialgebraic}
  if all of its generating functions are polynomials.
\end{definition}

Recall that a function is \emph{analytic}
 if and only if
 its Taylor series at $\mathbf{x}_0$
 converges to the function in some neighborhood
 for every $\mathbf{x}_0$ in its domain.
In the example of (\ref{eq:pathologicalIntersection}),
 ${\mathcal A}_s$ is semianalytic while 
 ${\mathcal A}_p$ is not,
 because the Taylor series of $g_2(x,y):=\sin\frac{1}{x}-y$
 at the origin does not converge.
Roughly speaking,
 the boundary curves of regular semianalytic sets
 are piecewise smooth.

As mentioned in the opening paragraph of \Cref{sec:motiv-from-solid},
 the B-rep of regular closed sets is not unique.
In addition, 
 the analysis of explicit IT methods
 via the theory of donating regions \cite{zhang13:_donat_region}
 also requires that the regular sets be open.
Therefore,
 only regular open semianalytic sets are employed in this work.

\begin{definition}
  \label{def:YinSet}
  A \emph{Yin set}\footnote{
    Yin sets are named after Qinghai Zhang's mentor, Madam Ping Yin.
    As a coincidence,
     the most important dichotomy in Taoism
     consists of Yin and Yang,
     where Yang represents the active, the straight, the ascending,
     and so on, 
     while Yin represents the passive, the circular, the descending,
     and so on.
  From this viewpoint, straight lines and Jordan curves can be 
   considered as Yang 1-manifolds and Yin 1-manifolds, respectively.
  } $\mathcal{Y}\subseteq \mathbb{R}^{\Dim}$ ($\Dim =2,3$)
   is a regular open semianalytic set
   whose boundary is bounded.
  The class of all such Yin sets form
   the \emph{Yin space} $\mathbb{Y}$.
\end{definition}

In this work we focus on the 3D case; 
 see \Cref{fig:yinsetexm} for examples of 3D Yin sets. 

 \begin{figure}
   \centering 
   \subfigure[each singular point is an isolated non-manifold point
   with unique tangent plane
   ]{
     \includegraphics[width=0.27\textwidth]{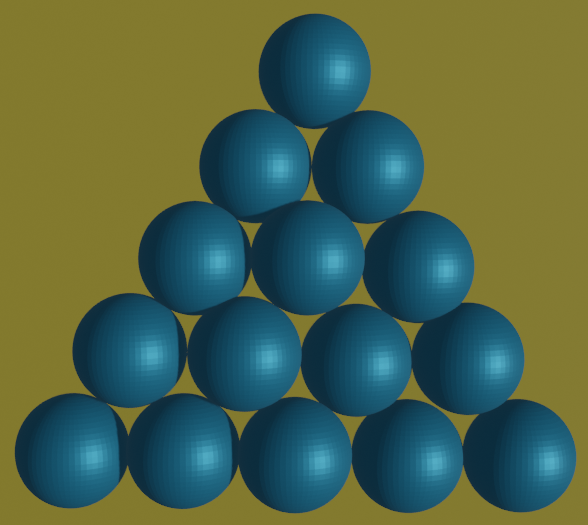}
   }
   \hfill
   \subfigure[all singular points are non-isolated non-manifold
   points in one connected component]{
     \includegraphics[width=0.36\textwidth]{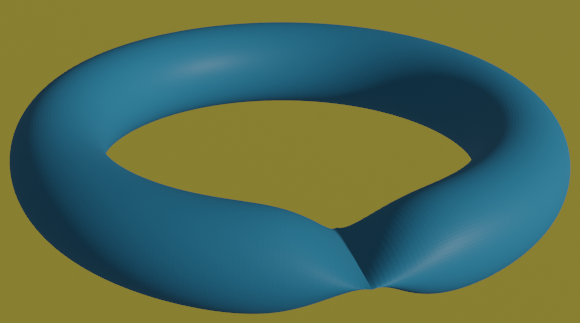}
   }
   \hfill
   \subfigure[all non-manifold points are non-isolated
   and in multiple connected components]{
     \includegraphics[width=0.27\textwidth]{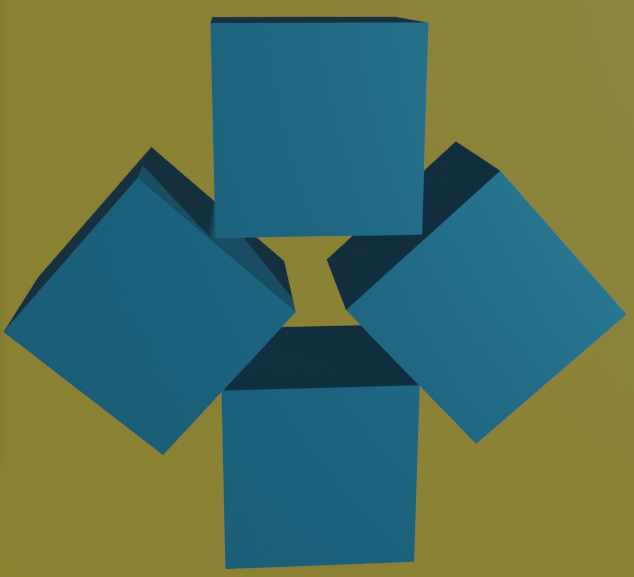}
   }

   \subfigure[all non-isolated non-manifold points
   are in a single connected component
   consisting of two intersecting curves (top view)]{
     \includegraphics[width=0.31\textwidth]{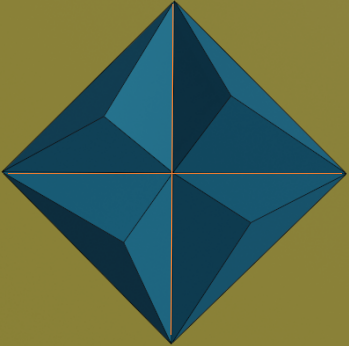}
   }
   \hfill
   \subfigure[the bottom view of the Yin set in (d);
   the four endpoints and the intersection of
   the two yellow lines are vertex non-manifold points
   ]{
     \includegraphics[width=0.32\textwidth]{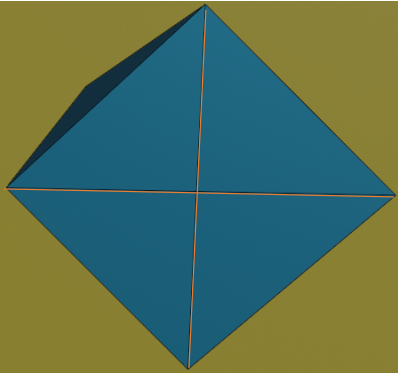}
   }
   \hfill
   \subfigure[all singular points are non-manifold, non-isolated, 
   and in two connected components, 
   each of which is a Jordan curve]{
     \includegraphics[height=0.29\textwidth]{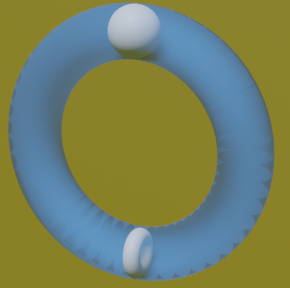}}
   \caption{Examples of 3D Yin sets and singular points on their boundaries.
     The Yin set in (d) consists of four tetrahedrons
     whose bottom edges touch upon each other along the two yellow lines. 
     Each connected component of a Yin set in (a--e)
     is homeomorphic to the unit open ball
     while the Yin set in (f) is not: 
     it is obtained by removing
     a sphere and a small torus from the big solid torus.
     The Yin sets in (a,b,f) only have non-manifold singular points
     while those in (c-e) also contain kinks. 
  }
  \label{fig:yinsetexm}
 \end{figure}


\section{Singular points of a Yin set}
\label{sec:singularPoints}

In this section, we capture all characteristic points
 of a Yin set in the definition of singular points
 and give several different dichotomies of them 
 to facilitate the analysis and the unique B-rep
 in later sections.
 
\begin{definition}[Manifold/non-manifold points]
  \label{def:manifoldPoints}
  A boundary point $p$ of a Yin set ${\cal Y}\subset \mathbb{R}^3$
  is called a \emph{manifold point} of $\partial {\cal Y}$ 
  if there exists a sufficiently small open ball
  ${\cal N}(p)\subset \mathbb{R}^3$ centered at $p$
  such that ${\cal N}(p)\cap \partial{\cal Y}$
  is homeomorphic to the unit open disk in $\mathbb{R}^2$;
  otherwise $p$ is a \emph{non-manifold point}
  of $\partial {\cal Y}$. 
\end{definition}

\begin{definition}[Singular/non-singular points]
  \label{def:singularPoints}
  A boundary point $p$ of a Yin set ${\cal Y}\subset \mathbb{R}^3$
  is \emph{singular}
  if $p$ is a non-manifold point
  or the tangent plane of $\partial {\cal Y}$ at $p$
  is not uniquely defined; 
  otherwise $p$ is \emph{non-singular}.
\end{definition}

By \Cref{def:singularPoints,def:manifoldPoints},
 a non-manifold point is always a singular point; 
 conversely,
 a singular point is either a non-manifold point or a \emph{kink},
 i.e., a manifold point of $\partial {\cal Y}$
 with ${\cal C}^1$ discontinuity. 
 
\begin{definition}[Isolated/non-isolated non-manifold points]
  \label{def:isolatedSingular}
  A non-manifold point $p\in \partial {\cal Y}$ is \emph{isolated}
  if $p$ is the only non-manifold point
  in ${\cal N}(p)\cap \partial{\cal Y}$
  for some sufficiently small open ball
  ${\cal N}(p)\subset \mathbb{R}^3$ centered at $p$;
  otherwise $p$ is \emph{non-isolated}.
\end{definition}

\begin{definition}[Vertex/non-vertex non-manifold points]
  \label{def:VertexSingularPoint}
  A non-isolated non-manifold point $p$ 
  of a Yin set ${\cal Y}$
  is said to be \emph{non-vertex non-manifold}
  if there exists a sufficiently small ${\cal N}(p)\subset \mathbb{R}^3$
  such that the set of all non-isolated non-manifold points 
  in ${\cal N}(p)\cap\partial {\cal Y}$
  is homeomorphic to $(0,1)$; 
  otherwise $p$ is \emph{vertex non-manifold}.
\end{definition}

In a sufficiently small neighborhood of a vertex non-manifold point $p$, 
 the set of non-isolated non-manifold points
 is homeomorphic either to $[0,1)$ or to a \emph{star},
 i.e., the union of three or more curves
 which intersect at their common endpoint $p$
 and are otherwise pairwise disjoint. 

\begin{figure}
  \centering 
  \includegraphics[width=0.8\textwidth]{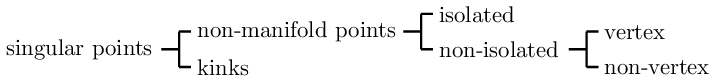}
  \caption{The relations of various singular points.
  }
  \label{fig:relationsOfSingularPoints}
\end{figure}

By \Cref{def:isolatedSingular,def:VertexSingularPoint}, 
 each non-manifold point of a Yin set
 is either isolated or non-isolated;
 each non-isolated non-manifold point
 is either vertex non-manifold or non-vertex non-manifold.
The three sets of
 isolated, vertex, or non-vertex non-manifold points
 are pairwise disjoint; 
 see \Cref{fig:relationsOfSingularPoints}
 for relations of these singular points
 and \Cref{fig:yinsetexm} for some examples.

In a similar manner, 
 we can also define isolated, vertex, and non-vertex kinks, 
 but their topology is usually simpler than
 (or at most similar with)
 that of non-manifold points.
Therefore in this work we do not study them except in \Cref{lem:CW-complex}.



\section{A surface is a piecewise-smooth orientable compact 2-manifold}
\label{sec:compact-2-manifolds}

An \emph{$m$-manifold} is a second countable Hausdorff space 
 that is locally homeomorphic to $\mathbb{R}^m$.
For example, a 1-manifold could be a Jordan curve, a line,
 or a Jordan curve minus one point; 
a 2-manifold could be a simple closed surface
 or an \emph{open surface patch}, i.e.,
 the homeomorphic image of $(0,1)^2$ under a continuous map
 $(0,1)^2\mapsto \mathbb{R}^3$.
An \emph{$m$-manifold with boundary}
 is the union of an $m$-manifold ${\cal S}$
 and an $(m-1)$-manifold that is the boundary of ${\cal S}$. 
Consider the \emph{$n$-disk} and the \emph{$n$-sphere}, 
 \begin{equation}
   \label{eq:nDiskAndnSphere}
   \mathbb{D}^{n}:=
   \{\mathbf{x}\in\mathbb{R}^{n}: \|\mathbf{x}\|\le 1\};
   \quad
   \mathbb{S}^{n}:=\{\mathbf{x}\in\mathbb{R}^{n+1}: \|\mathbf{x}\|=1\}.
 \end{equation}
$\mathbb{D}^2$ is a 2-manifold with boundary
 and $\mathbb{S}^1$ is the 1-manifold boundary of $\mathbb{D}^2$.


For a smooth compact 2-manifold ${\cal S}$, 
 the tangent space $T_p$ of $p\in{\cal S}$ 
 is a 2D real vector space.  
Two bases $(\mathbf{v}_1,\mathbf{v}_2)$ and $(\mathbf{u}_1,\mathbf{u}_2)$ of $T_p$
 \emph{define the same orientation}
 if the sign of the determinant of the change-of-basis matrix
 from $(\mathbf{v}_1,\mathbf{v}_2)$ to $(\mathbf{u}_1,\mathbf{u}_2)$ is $+$; 
 otherwise they \emph{define different orientations}.
This binary relation of ``defining the same orientation'' 
 is an equivalence relation on all bases of $T_p$ 
 and each of the two equivalence classes
 is called an \emph{orientation of the tangent space} $T_p$.
The \emph{orientation of a point} $p\in {\cal S}$
 is identified with that of $T_p$.

A smooth compact 2-manifold ${\cal S}$
 is \emph{orientable} if each point in ${\cal S}$
 can be assigned an orientation such that
 any two points sufficiently close have the same orientation.
A piecewise-smooth compact 2-manifold ${\cal S}$
 is \emph{orientable} if
 it can be approximated arbitrarily well
 by a smooth orientable compact 2-manifold.

\begin{definition}[Surface]
  \label{def:surface}
  A \emph{surface} is a piecewise-smooth orientable compact 2-manifold 
  and a \emph{surface with boundary}
  is a piecewise-smooth orientable compact 2-manifold with boundary.
\end{definition}

The following concepts of proper and improper intersections
 are crucial in the unique B-rep of Yin sets,
 cf. \Cref{def:almostDisjointGsurf};
 also see \Cref{fig:intersections} for an illustration. 

\begin{definition}[Intersections of surfaces and/or surfaces with boundary]
  \label{def:intersectionsOfGluedSurfs}
  Let $\mathcal{S}_1$ and $\mathcal{S}_2$ be
  two surfaces or surfaces with boundary; 
  they are \emph{disjoint}
  if $\mathcal{S}_1\cap \mathcal{S}_2=\emptyset$. 
  Otherwise a maximally connected component $\gamma$ of
  $\mathcal{S}_1\cap \mathcal{S}_2$ 
  is called an \emph{improper intersection}
  of $\mathcal{S}_1$ and $\mathcal{S}_2$
  if $(\mathcal{S}_2\cap \mathcal{N}(\gamma)) \backslash \gamma$
  is contained in a single component of
  $\mathcal{N}(\gamma) \backslash \mathcal{S}_{1}$
  in a sufficiently small open neighborhood
  $\mathcal{N}(\gamma)\subset \mathbb{R}^3$;
  otherwise $\gamma$ is a \emph{proper intersection}.
\end{definition}

\begin{figure}
  \centering 
  \subfigure[${\cal S}_1$ and ${\cal S}_2$ cross each other
  at a proper intersection]{
    \includegraphics[width=0.45\textwidth]{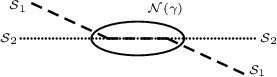}
  }
  \hfill
  \subfigure[${\cal S}_1$ and ${\cal S}_2$ does not cross each other
  at an improper intersection]{
    \includegraphics[width=0.45\textwidth]{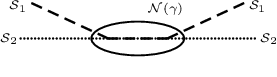}
  }
  \caption{Examples of proper and improper intersections
    in \Cref{def:intersectionsOfGluedSurfs}.
  }
  \label{fig:intersections}
\end{figure}

The following celebrated theorem 
 is one of the early triumphs of algebraic topology;
 see \cite[Chap. 1]{massey77:_alg_topo_an_intro} for a proof.
 
\begin{theorem}[Classification of surfaces]
  \label{thm:manifoldclassification}
  Every surface is homeomorphic
  to a 2-sphere, a torus, or a connected sum of tori.
\end{theorem}

A simple closed curve is homeomorphic to $\mathbb{S}^1$
 and its complement in $\mathbb{R}^{2}$
 consists of two connected components: 
 one bounded and the other unbounded;
 this is the classical Jordan curve theorem
 \cite{jordan87:_cours_daanl_polyt}. 
For $n \ge 2$, the Jordan--Brouwer separation theorem states that
 the complement of $\mathbb{S}^{n-1}$ in $\mathbb{R}^{n}$
 also consists of two connected components: 
 one bounded and the other unbounded;
 see \cite[Theorem 6.35]{rotman88:_intro_to_alge_topo}.
In the particular case of $n = 3$,
 the conclusion of the Jordan--Brouwer separation theorem
 also holds for a surface
 \cite{lima88:_jordan_brouwer_separation_theorem_for_hype,
  _schmaltz09:_jordan_brouwer_separation_theorem}.
This dichotomy of surface complement
 into bounded and unbounded components
 is of fundamental importance
 in defining the key concept of glued surfaces
 in \Cref{def:gluedsurface}.
 
\begin{definition}[Orientations of a surface]
  \label{def:orientationOfSurface}
  For a surface ${\cal S}$ whose orientation is indicated
  by a basis $(\mathbf{v}_1,\mathbf{v}_2)$ of the tangent space $T_p{\cal S}$,
  we regard $\mathbf{v}_1$ and $\mathbf{v}_2$ as vectors in $\mathbb{R}^3$
  by the embedding of $T_p{\cal S}$ in $T_p\mathbb{R}^3$
  and define the \emph{induced normal vector} $\mathbf{n}$ at $p\in{\cal S}$
  as the vector 
  induced from the right-hand rule
  to form a basis $(\mathbf{v}_1,\mathbf{v}_2, \mathbf{n})$ of $\mathbb{R}^3$. 
  The surface ${\cal S}$ is \emph{positively oriented}
  if $\mathbf{n}$ 
  always points from its bounded complement 
  to its unbounded complement;
  otherwise it is \emph{negatively oriented}.
\end{definition}

\begin{definition}[Internal complement of a surface]
  \label{def:interiorOfSurface}
  The \emph{internal complement of a surface ${\cal S}$},
  written $\mathrm{int}({\cal S})$, 
  is its bounded complement
  if ${\cal S}$ is positively oriented
  and its unbounded complement
  if ${\cal S}$ is negatively oriented.
\end{definition}

\Cref{def:orientationOfSurface,def:interiorOfSurface}
pave the way to the unique B-rep
in \Cref{thm:connectedYinsetrep}.



\section{Local topology}
\label{sec:localTopology}

For an interior point or a manifold point of a Yin set,
 its local neighborhood has the trivial topology
 of a ball or a half ball, respectively. 
So we are mostly interested
 in the local topology of non-manifold points.
 
\subsection{The good neighborhood of a boundary point}
\label{sec:topology-at-singular}

We start with 

\begin{definition}[Generalized disk]
  \label{def:genDisk}
  For a Yin set ${\cal Y}$, 
  a \emph{generalized disk ${\cal D}$ at a boundary point} $p\in \partial {\cal Y}$
  within an open ball
  \mbox{$\mathcal{N}(p)\subset \mathbb{R}^3$} 
  is the homeomorphic image of
  $\mathbb{D}^2$ in (\ref{eq:nDiskAndnSphere})
  under a continuous map
  $f: \mathbb{D}^2 \rightarrow
  \partial {\cal Y} \cap \overline{\mathcal{N}(p)}$,
  such that $p=f(\mathbf{0})$
  and $f(\mathbb{S}^1)\subset \partial {\cal N}(p)$,
  where $\mathbb{S}^1$ is the boundary of $\mathbb{D}^2$
  and $f(\mathbb{S}^1)$ is called
  the \emph{boundary of the generalized disk},
  written $\partial {\cal D}$.
\end{definition}

In \Cref{fig:yinsetexm}(a),
 each isolated singular point is the intersection
 of two generalized disks. 
By \Cref{def:isolatedSingular,def:genDisk},
 a boundary point $p\in \partial {\cal Y}$ is non-singular
 if and only if $p$ is not a kink
 and the number of generalized disks at $p$ is 1.

\begin{definition}[Generalized radius]
  \label{def:genRadius}
  For a Yin set ${\cal Y}$, 
  a \emph{generalized radius of a boundary point} $p\in \partial {\cal Y}$
  within an open ball
  \mbox{$\mathcal{N}(p)\subset \mathbb{R}^3$} centered at $p$
  is a simple curve $\gamma\subset\partial {\cal Y}$
  from $p$ to some point in $\partial {\cal Y}\cap\partial \mathcal{N}(p)$.
\end{definition} 

\begin{figure}
  \centering 
  \subfigure[$p$ is a non-vertex non-manifold point]{
    \includegraphics[width=0.43\textwidth]{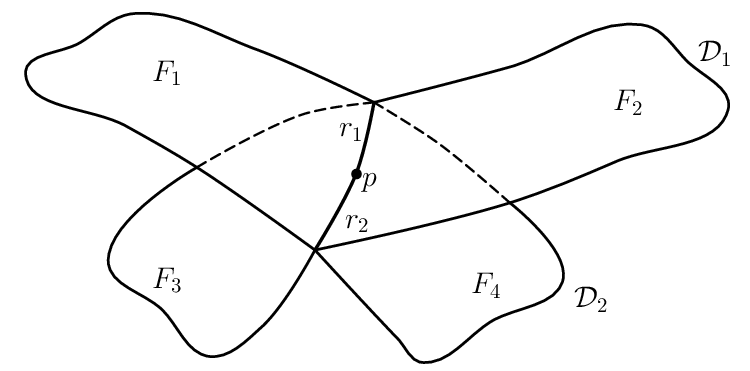}
  }
  \hfill
  \subfigure[$p$ is a vertex non-manifold point]{
    \includegraphics[width=0.43\textwidth]{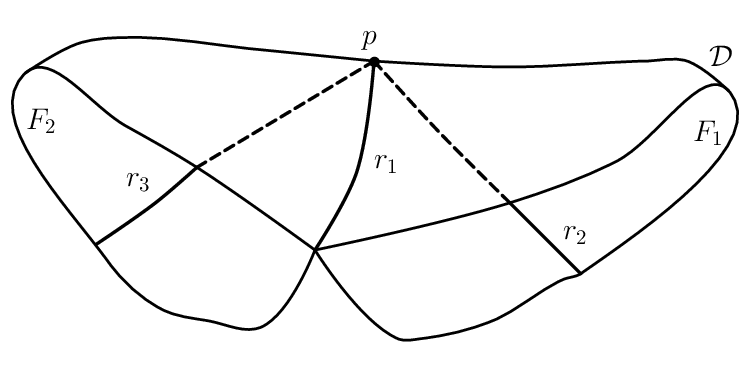}
  }
  \caption{Examples of generalized disks, generalized radii, 
    and generalized sectors.
    In (a), each of the four sectors is on both $r_1$ and $r_2$;
    each of the two disks is the union of two sectors: 
    ${\cal D}_1=F_1\cup F_2$ and ${\cal D}_2=F_3\cup F_4$.
    In (b), $\partial {\cal Y}\cap \overline{{\cal N}(p)}$
    is a folded disk in \Cref{def:foldedDisk}
    with $k=2$ and $t=1$.
    Originally,
    both of the two sectors $F_1$ and $F_2$
    with ${\cal D}=F_1\cup F_2$
    are on the radius $r_1$; 
    the artificial radii $r_2$ and $r_3$
    are added according to \Cref{cvn:evenNumberOfSectors}
    so that each of $F_1$ and $F_2$
    are on two radii.
  }
  \label{fig:disksAndRadii}
\end{figure}

The main purpose of a generalized radius is to represent
 (part of) a connected component of non-isolated non-manifold points.
In \Cref{fig:yinsetexm}(b),
 all non-manifold points are non-isolated
 and they constitute a single simple curve
 $\gamma:[0,1]\rightarrow \partial {\cal Y}$, 
 of which the endpoints are vertex non-manifold
 and all other points are non-vertex non-manifold. 
For each non-vertex non-manifold point $p$,
 the set of all non-isolated non-manifold points
 in a sufficiently small neighborhood ${\cal N}(p)$
 is the union of two generalized radii
 which intersect only at $p$;
 see \Cref{fig:disksAndRadii}(a). 
In contrast, for a vertex non-manifold point $p_{v}$ of $\gamma$, 
 a single generalized radius $r_v$ at $p_{v}$
 might contain all non-isolated non-manifold points
 in ${\cal N}(p_{v})\cap \partial {\cal Y}$;
 see \Cref{fig:disksAndRadii}(b). 

In \Cref{fig:yinsetexm}(c),
 non-manifold points 
 constitute four pairwise distinct simple curves.
In the neighborhood of a vertex non-manifold point
 $q_{\gamma}$ of each curve $\gamma$, 
 the set of all non-isolated non-manifold points
 is a single generalized radius $r_q$ at $q_{\gamma}$.

\begin{definition}[Generalized sector]
  \label{def:genSector}
  For a boundary point $p\in \partial {\cal Y}$
  of a Yin set ${\cal Y}$, 
  two distinct generalized radii $r_1$ and $r_2$ of $p$
  in a generalized disk ${\cal D}$ centered at $p$
  cut ${\cal D}$ into two connected components
  whose boundaries relative to ${\cal D}$ are both $r_1\cup r_2$.
  The closure of each of the two connected components is
   called a \emph{generalized sector}
   with $r_1$ and $r_2$ as the generalized radii \emph{of} this sector.
  If $r$ is a generalized radius of a sector $F$,
    then $F$ is called a generalized sector \emph{on} $r$.
\end{definition}

In \Cref{fig:yinsetexm}(d,e), 
 the single connected component of non-manifold points 
 can be considered as the union of two line segments intersecting only
 at a vertex non-manifold point,
 whose local neighborhood contains four generalized radii
 and eight generalized sectors,
 with four generalized sectors on each generalized radius.

In the aforementioned example of \Cref{fig:yinsetexm}(c),
 there exists a sector $F$ in ${\cal N}(q_{\gamma})$
 of which both generalized radii are $r_q$.
Similarly, in \Cref{fig:yinsetexm}(b),
 each of the two sectors in ${\cal N}(p_{v})$
 is on $r_v$ and only on $r_v$.
These special cases motivate 

\begin{cvn}
  \label{cvn:evenNumberOfSectors}
  Whenever a sector $F$ is on a single radius $r$,
  we pick another radius $r'$ in $F$
  intersecting $r$ only at $p$,
  so that $r$ and $r'$ divide $F$ into two sectors.
\end{cvn}

The ``artificial'' generalized radius $r'$
 is different from $r$ in that
 $r'\setminus \{p\}$ consists of only manifold points of ${\cal Y}$.
\Cref{cvn:evenNumberOfSectors} will be convenient in later analysis 
 since, although the number of sectors
 may vary from one good neighborhood to another, 
 the number of sectors on each radius $r$ is always even. 

All the aforementioned examples in \Cref{fig:yinsetexm} are special cases of
 
\begin{lemma}[Local topology of a boundary point]
  \label{lem:localTopoOfYinSets}
  For a boundary point $p$ 
  of a 3D Yin set ${\cal Y}$,
  any sufficiently small open neighborhood $\mathcal{N}(p)\subset \mathbb{R}^3$
  satisfies 
  \begin{enumerate}[(a)]
  \item $\partial {\cal Y}\cap \overline{\mathcal{N}(p)}$
    is the union of finitely many generalized disks;
  \item Pairwise intersection of the generalized disks in (a) is either
    $p$ itself or the union of finitely many generalized radii pairwise intersecting only at $p$;
  \item  $\mathcal{N}(p)\setminus \partial {\cal Y}$
    consists of disjoint regular open sets;
    two such sets sharing a common sector as part of their boundary
    are subsets of ${\cal Y}$ and ${\cal Y}^{\perp}$, respectively.
  \end{enumerate}
\end{lemma}
\begin{proof}
  By \Cref{def:semianalytic-sets} and ${\cal Y}$ being semianalytic,
  $\partial {\cal Y}\cap \mathcal{N}(p)$
  is defined by a finite number of analytic functions
  $g_i:\mathbb{R}^3\rightarrow\mathbb{R}$,
  each of which can be replaced by its Taylor series expanded at $p$.
  Thus each $g_i$ can be approximated arbitrarily well,
  within a sufficiently small open neighborhood $\mathcal{N}(p)$,
  by a linear function of the form
  $\tilde{g}_i(x,y,z)=a_0+a_1x+a_2y+a_3z$.
  Hence the points satisfying $g_i(x,y,z)=0$ can be regarded
  as the graph of some linear height function \mbox{$\chi=H_i(\xi,\eta)$}.
  Therefore, $\partial {\cal Y}\cap \overline{\mathcal{N}(p)}$ can only contain
  the homeomorphic images of $\mathbb{D}^2$,
  which, by \Cref{def:genDisk}, are generalized disks.
  Therefore (a) holds.
  
  The intersection of any two distinct generalized disks in (a)
   contains no 2D components;
   otherwise it would contradict
   \Cref{def:YinSet}
   that any Yin set is regular open.
  By \Cref{def:genDisk}, each generalized disk
   is a surface with boundary
   and is thus triangularizable. 
  Consequently, ${\cal Y}\cap {\cal N}(p)$
   can be approximated arbitrarily well,
   within a sufficiently small neighborhood ${\cal N}(p)$, 
   by a semianalytic set whose generating functions
   are multivariate linear polynomials
   corresponding to triangles in the triangular mesh.
  The intersection of two such triangular meshes
   is either $p$ itself or
   a finite number of curves intersecting only at $p$.
  Therefore (b) holds.
    
  By (a) and (b),
   we get a finite number of disjoint regular open sets
   after deleting from ${\mathcal N}(p)$ a finite number of surfaces.
  For two such regular open sets sharing a common sector, 
  both cannot be contained in ${\cal Y}$ 
  because this would contradict ${\cal Y}$ being regular open.
  Similarly, both cannot be contained in ${\cal Y}^{\perp}$ either. 
  Hence (c) holds.
\end{proof}

\begin{definition}[Good neighborhood]
  \label{def:goodNeighborhood}
  A \emph{good neighborhood of a boundary point} $p\in \partial {\cal Y}$
  is a sufficiently small open set $\mathcal{N}(p)\supset p$
  satisfying (i) ${\cal N}(p)$ is homeomorphic to the interior of
  $\mathbb{D}^3$
  and (ii) the conclusions of Lemma \ref{lem:localTopoOfYinSets}
  hold in $\mathcal{N}(p)$. 
\end{definition}

By \Cref{lem:localTopoOfYinSets}, 
 the type of a boundary point $p\in \partial {\cal Y}$
 and the topology of $\partial {\cal Y}\cap{\cal N}(p)$
 are independent of the choice of
 the good neighborhood $\mathcal{N}(p)$.

\begin{ntn}
  Hereafter a good neighborhood of a singular point $p$
  will be denoted by $\mathcal{N}(p)$. 
  A \emph{disk} refers to a generalized disk
  in $\partial {\cal Y} \cap\overline{\mathcal{N}(p)}$
  given by Lemma \ref{lem:localTopoOfYinSets}(a), 
  a \emph{radius} refers to a generalized radius
  along which \emph{multiple} disks in $\overline{\mathcal{N}(p)}$ intersect,
  cf. \Cref{def:genRadius}, 
  and  a \emph{sector} refers to a generalized sector
  on a radius as in \Cref{def:genSector}. 
  The \emph{collection of all radii} in $\overline{\mathcal{N}(p)}$
  is denoted by $\mathcal{R}(p)$.
\end{ntn}
 %


\subsection{The global topology of singular points}

The \emph{disjoint union} ${\cal U}_{\cal X}$
 of a family $\{{\cal X}_{\alpha}\}_{\alpha\in J}$
 of topological spaces is given by 
\begin{equation}
  \label{eq:topologicalSum}
  {\cal U}_{\cal X} := \sqcup_{\alpha\in J} {\cal X}_{\alpha}
  :=\cup_{\alpha\in J} {\cal U}_{\alpha} \
  \text{ where }
  {\cal U}_{\alpha} := {\cal X}_{\alpha}\times \{\alpha\}.
\end{equation}
If we topologize ${\cal U}_{\cal X}$ by declaring $U$ to be open in ${\cal U}_{\cal X}$
if and only if $U\cap {\cal U}_{\alpha}$ is open in ${\cal U}_{\alpha}$
for each $\alpha$,
then ${\cal U}_{\cal X}$ is said to be the \emph{topological sum} of the spaces
${\cal U}_{\alpha}$ or ${\cal X}_{\alpha}$.

\begin{definition}[Attaching space/map]
  \label{def:attaching-spaces}
  For two topological spaces ${\cal X}_1$, ${\cal X}_2$,
  and a continuous function $f: B \rightarrow A$
  with $A\subset {\cal X}_1$ and $B\subset {\cal X}_2$,
  the \emph{attaching space} of ${\cal X}_1$ and ${\cal X}_2$,
  written ${\cal X}_1\sqcup_f {\cal X}_2$, 
  is the quotient space of the topological sum ${\cal X}_1\sqcup {\cal X}_2$
  obtained by identifying each $a\in A$ with every $b\in f^{-1}(a)$.
  Then the continuous function $f$ is called
  the \emph{attaching map} of ${\cal X}_1\sqcup_f {\cal X}_2$. 
\end{definition}

In this work the attaching map is assumed to be well behaved
 so that the attaching space always has the quotient topology.
For example,
 let ${\cal X}_2=\mathbb{D}^2$, $B=\partial {\cal X}_2$,
 and ${\cal X}_1=A$ be the space of one point. 
For the attaching map $f: B\rightarrow A$
 that collapses $B$ to the single point in $A$,
 the attaching space ${\cal X}_1\sqcup_f {\cal X}_2$
 is homeomorphic to $\mathbb{S}^2$.

\begin{definition}[1D CW complex]
  A \emph{one-dimensional (1D) CW complex} $K_{1}$
  is an attaching space
  constructed from 0-cells (points)
  $p_1, p_2, \ldots,  p_{k_0}$, 
  1-cells (curve segments homeomorphic to $\mathbb{D}^1$)
  $\gamma_1, \gamma_2, \ldots,  \gamma_{k_1}$
  and an attaching map $f$ as follows.
  \begin{enumerate}[(a)]
  \item The \emph{0-skeleton} of $K_1$ is
    the disjoint union of its 0-cells,
    written $K^{(0)}:= \sqcup_{i=1}^{k_0} p_i$; 
  \item Build the \emph{1-skeleton} as
    an attaching space 
    $K^{(1)} := C^{1} \bigsqcup\nolimits_{f} K^{(0)}$,
    where the attaching map has the signature
    $f: \partial C^1 \to K^{(0)}$
    and $C^{1} := \sqcup_{j=1}^{k_1} \gamma_j$; 
  \item The 1-skeleton is taken to be the cell complex,
    i.e., $K_{1}:= K^{(1)}$.
  \end{enumerate}
\end{definition}
  
A 1D CW complex is always homeomorphic to a graph $G=(V,E,\phi_G)$, 
 where the incidence function $\phi_G: E\rightarrow V\times V$
 can be deduced directly from the attaching map.
When the degree of each vertex in $V$ is either 2 or 1, 
 a 1D CW complex is either a 1-manifold or a 1-manifold with boundary; 
 otherwise it is not a 1-manifold.
For a Yin set shown in \Cref{fig:yinsetexm},
 a connected component of its singular points
 may be one isolated singular point (subplot (a)),
 a curve segment (subplot (b)),
 a 1-manifold (subplot (f)),
 or some 1D subset that is neither a 1-manifold nor a 1-manifold with boundary
 (subplots (c,d,e)).
All these examples are special cases of
 
\begin{lemma}
  \label{lem:CW-complex}
  Any Yin set ${\cal Y}\subset \mathbb{R}^3$ satisfies 
  \begin{enumerate}[(a)]
  \item $\partial {\cal Y}$ contains finitely many isolated singular points;
  \item the non-manifold points of $\partial {\cal Y}$
    form a 1D CW complex,
    so do the manifold singular points of $\partial {\cal Y}$. 
  \end{enumerate}
\end{lemma}
\begin{proof}
  By \Cref{def:YinSet}, $\partial {\cal Y}$ is compact.
  Suppose $\partial {\cal Y}$ contains
  infinitely many isolated singular points.
  Then there must exist a point $y\in \partial {\cal Y}$
  such that any neighborhood $\partial {\cal Y}\cap {\cal N}(y)$
  contains infinitely many isolated singular points.
  By \Cref{def:semianalytic-sets},
   none of the neighborhood of $y$ is a good neighborhood,
   and this contradicts Lemma \ref{lem:localTopoOfYinSets}.
  Therefore (a) holds.
  As for (b), we only consider the case of non-manifold points
  since the proof for manifold singular points is similar.

  Let $l$ be a connected component of the set of all non-manifold
  singular points.
  For any $p\in l$, 
  $l\cap {\cal N}(p)$ is homeomorphic
  to $\bigcup_{r\in \mathcal{R}(p)}r$,
  where all radii intersect only at $p$. 
  So $l$ is locally homeomorphic to a star in a graph.
  Therefore $l$ is homeomorphic to a graph
  and hence can be given the structure of a 1D CW complex.

  To show that the set of non-manifold points is closed in $\partial {\cal Y}$,
  we show that any convergent sequence of non-manifold points
  converges to a non-manifold  point.
  Suppose the point $q$ to which the sequence converges
  is not a non-manifold point.
  Then $q$ is a manifold point
  and \Cref{def:manifoldPoints} implies
  the existence of a good neighborhood
  of $q$ that contains none of the non-manifold points, 
  which contradicts the convergence of the sequence.
  Therefore the set of non-manifold points
  must be compact. 
\end{proof}

The homeomorphism between the 1D CW complex in \Cref{lem:CW-complex}(b)
 and a graph explains
 the names ``vertex non-manifold points''
 and ``non-vertex non-manifold points'' in \Cref{def:VertexSingularPoint}.

\subsection{Folded disks}
\label{sec:folded-disks}

By setting both ${\cal X}_1$ and ${\cal X}_2$ in \Cref{def:attaching-spaces}
 to a generalized disk and requiring that the attaching map
 preserves all radii, we get

\begin{definition}
\label{def:foldedDisk}  
  A \emph{folded disk} is the attaching space of a generalized disk 
  via an attaching map 
  $f: r_1\cup \cdots \cup r_k\rightarrow l_1\cup\cdots\cup l_t$
  such that (i) $r_1,\dots, r_k$ and $l_1, \dots, l_t$
  are radii in the disk
  and (ii) $f$ maps each $r_i$ homeomorphically to $l_j$ for some $j$.
\end{definition}

Consider the Yin set in \Cref{fig:yinsetexm}(d,e). 
For each endpoint $q_e$ of the horizontal or vertical yellow lines, 
 $\partial {\cal Y}\cap \overline{\mathcal{N}(q_e)}$
 is \emph{uniquely} represented as a single folded disk centered at $q_e$
 with $k=2$ and $t=1$; see \Cref{fig:disksAndRadii}(b). 
In particular, $\partial {\cal Y}\cap \overline{\mathcal{N}(q_e)}$
 \emph{cannot} be considered as two disks glued along
 a single radius at $q_e$:
 although in \Cref{fig:disksAndRadii}(b)
 cutting ${\cal D}$ along $r_1$ yields two disks,
 but according to \Cref{def:genDisk}, 
 none of them is a disk \emph{at} $p$. 
In contrast, 
 for the intersection $q_X$ of the horizontal and vertical yellow lines, 
 the representation of $\partial {\cal Y}\cap \overline{\mathcal{N}(q_X)}$
 is far from unique:
 it can be represented as four disks glued along four radii at $q_X$,
 or two folded disks (with $k=2$ and $t=1$) glued along two radii at $q_X$,
 or a folded disk (with $k=4$ and $t=2$) and a disk glued along two radii at $q_X$, etc. 
 
\subsection{Good pairing}
\label{sec:good-pairing}

The ambiguity caused by folded disks is not conducive
 to the unique B-rep for Yin sets. 
A key notion to resolve this ambiguity is 

\begin{definition}[Good pairing]
  \label{def:goodPairing}
  For a non-isolated non-manifold point \mbox{$p\in \partial {\cal Y}$}, 
  a \emph{good pairing of sectors on a single radius $r$ in} $\mathcal{N}(p)$
  is a decomposition of these sectors into pairs such that
  the two set unions $F\cup F'$ and $G\cup G'$
  of any two pairs $(F,F')$ and $(G,G')$
  have no proper intersections, 
  cf. \Cref{def:intersectionsOfGluedSurfs}.
  The good pairings for all radii in $\mathcal{N}(p)$
  constitute a \emph{good pairing (of all sectors) within} $\mathcal{N}(p)$.
\end{definition}

A good pairing of sectors on a radius $r$ in $\overline{\mathcal{N}(p)}$
 is said to \emph{locally bound} ${\cal Y}$ 
 if, for each pair $(F,F')$ in the good pairing,
 $F\cup F'$ is a subset of the boundary of a single maximally connected component
 of $\overline{\mathcal{N}(p)}\cap {\cal Y}$. 
A good pairing within $\overline{\mathcal{N}(p)}$
 is said to \emph{locally bound} ${\cal Y}$
 if the good pairing for \emph{each} radius in $\overline{\mathcal{N}(p)}$
 locally bounds ${\cal Y}$. 

\begin{lemma}[Good pairing for a single radius]
  \label{lem:goodPairingsRadius}
  For a non-manifold point \mbox{$p\in \partial{\cal Y}$},
  the number $n_{\textrm{gp}}$ of different good pairings
  for any radius $r$ in $\mathcal{N}(p)$
  is either 1 or 2.
  For $n_{\textrm{gp}}=1$, the good pairing locally
  bounds both ${\cal Y}$ and ${\cal Y}^{\perp}$.
  For $n_{\textrm{gp}}=2$, 
  one good pairing locally bounds ${\cal Y}$
  while the other locally bounds ${\cal Y}^{\perp}$.
\end{lemma}
\begin{proof}
  $n_{\textrm{gp}}=1$ 
  if there are only two sectors within ${\cal N}(p)$,
  i.e., if $p$ is 
  a vertex non-manifold point with
  an artificial radius $r'$ in \Cref{cvn:evenNumberOfSectors}.
  Otherwise \Cref{def:goodNeighborhood} implies that
  the number of sectors on any $r\in\mathcal{R}(p)$
  is an even number greater than 2.
  By \Cref{def:goodPairing},
  two sectors form a pair if and only if they are next to each other.
  Hence there exist two and only two different good pairings.
  The rest of the proof follows from \Cref{lem:localTopoOfYinSets}(c).
\end{proof}



To proceed from the good pairing for a single radius
 to that in the entire good neighborhood, 
 we need
 
\begin{lemma}
  \label{lem:one2one}
 For two distinct non-isolated non-manifold points $p,q\in\partial {\cal Y}$, 
 suppose a radius $r_1$ in $\mathcal{N}(p)$
 has a non-empty intersection with a radius $r_2$ in $\mathcal{N}(q)$.
 Then there is a one-to-one correspondence
 between the sectors of $\mathcal{N}(p)$ on $r_1$
 and sectors of $\mathcal{N}(q)$ on $r_2$,
 where the corresponding sectors intersect in a 2D subset
 of $\mathcal{N}(p)\cap \mathcal{N}(q)$. 
\end{lemma}
\begin{proof}
  Let $x\in \mathcal{N}(p)$ be a point in $r_1\cap r_2$.
  Consider a good disk decomposition of a good neighborhood
  $\mathcal{N}(x)$. By making $\mathcal{N}(x)$ sufficiently small,
  we can assume that $\mathcal{N}(x)$ is contained in both
  $\mathcal{N}(p)$ and $\mathcal{N}(q)$.
  Note that $\mathcal{N}(x)$ has only two radii,
  each of which is a subset of $r_1\cap r_2$
  as long as we make $\mathcal{N}(x)$ small enough.
  Each sector in $\mathcal{N}(x)$ is contained in a sector of
  $\mathcal{N}(p)$ on $r_1$ and each sector of $\mathcal{N}(p)$ on
  $r_1$ contains exactly one sector of $\mathcal{N}(x)$.
  A similar statement holds for sectors in $\mathcal{N}(x)$ and
  $\mathcal{N}(q)$.
  Therefore, there is a natural one-one correspondence
  between sectors in $\mathcal{N}(p)$ and $\mathcal{N}(q)$
  in that the corresponding pair
  contains the same sector of $\mathcal{N}(x)$.  
\end{proof}

So long as ${\cal N}(p)$ and ${\cal N}(q)$ are good neighborhoods,
 the two distinct non-manifold points $p,q$ in \Cref{lem:one2one}
 can be both vertex non-manifold, or both non-vertex non-manifold, 
 or one vertex non-manifold and the other non-vertex non-manifold.

\begin{lemma}[Good pairings in a good neighborhood 
  of a non-manifold point]
  \label{lem:goodPairingsInNeighborhood}
  For any non-manifold point \mbox{$p\in \partial {\cal Y}$},
  the number $n_{\textrm{gp}}$ of different good pairings in ${\cal N}(p)$
  is either 1 or 2.
  If $n_{\textrm{gp}}=1$, the good pairing locally
  bounds both ${\cal Y}$ and ${\cal Y}^{\perp}$.
  If $n_{\textrm{gp}}=2$, 
  one good pairing locally bounds ${\cal Y}$
  while the other locally bounds ${\cal Y}^{\perp}$.
\end{lemma}
\begin{proof}
  First, we consider the case of $p$ being non-vertex non-manifold. 
  By \Cref{def:VertexSingularPoint,def:genRadius},
  $\mathcal{N}(p)$ has exactly two radii $r_1$ and $r_2$
  intersecting only at $p$
  and the number of sectors on each of the two radii
  is 2$n$ where $n>1$.
  Choose two non-vertex non-manifold points $q_1\in r_1$, $q_2\in r_2$
  and two good neighborhoods ${\cal N}(q_1)$, ${\cal N}(q_2)$
  so that ${\cal N}(q_1)\subset {\cal N}(p)$,
  ${\cal N}(q_2)\subset {\cal N}(p)$,
  and ${\cal N}(q_1)\cap {\cal N}(q_2)\ne \emptyset$.
  Then the conclusion follows from \Cref{lem:one2one}. 

  Hereafter we consider the case of $p$ being vertex non-manifold. 

  Suppose $\overline{{\cal N}(p)}\cap \partial {\cal Y}$ can be uniquely represented
  as a folded disk with $t=1$, 
  cf. the example under \Cref{def:foldedDisk} with $k=2$ and $t=1$. 
  Then by \Cref{def:foldedDisk},
  all sectors share the common radius $l_1$.
  By \Cref{cvn:evenNumberOfSectors}, all other radii are artificial.
  Therefore there is only one good pairing of these sectors.

  Otherwise $\overline{{\cal N}(p)}\cap \partial {\cal Y}$
  is either a folded disk with $t>1$ or not a folded disk.
  In both cases,
  \Cref{lem:CW-complex} dictates
  that the set ${\cal I}$ of all non-isolated non-manifold points is a 1D CW complex
  homeomorphic to a graph $G=(V, E, \phi_G)$,
  where $V$ is the set of all vertex non-manifold points
  and $E$ that of curves connecting points in $V$.
  Choose a set $P:=\{p_1,p_2, \ldots, p_m\}$ of points in $E$
  and a good neighborhood for each point in $P\cup V$
  such that these good neighborhoods cover ${\cal I}$
  and the intersection of any two good neighborhoods
  is either the empty set or a 3-manifold.
  Then the proof is completed by applying \Cref{lem:one2one}
  to all pairs of intersecting good neighborhoods. 
\end{proof}

\subsection{Good disk decomposition}
\label{sec:good-disk-decomp}

\Cref{lem:goodPairingsRadius} leads to another characterization
 of the local topology of a non-manifold point
 in terms of disks and folded disks.
 
\begin{lemma}
  \label{lem:diskDecomposition}
  For a non-manifold point $p\in \partial {\cal Y}$, 
  the 2D subset $\partial {\cal Y}\cap \overline{\mathcal{N}(p)}$
  is the union of a finite number of disks and folded disks,
  no pair of which intersect properly. 
  In particular,
  $p$ being non-vertex non-manifold implies that
  $\partial {\cal Y}\cap \overline{\mathcal{N}(p)}$
  is the union of a finite number of disks,
  each of which is the union of a good pair of sectors.
\end{lemma}
\begin{proof}
  By \Cref{lem:goodPairingsRadius}, for each $r\in \mathcal{R}(p)$, 
  we always have a good pairing of sectors on $r$.
  Then the sectors can be arranged to form disks or folded disks:
  for any initial sector $F_0$,  
  the pairing gives a unique sequence
  of sectors $F_0, F_1, \dots, F_{k-1}, F_{k}=F_0$,
  where each sector $F_i$ is on radii $r_i$ and $r_{i+1}$,
  the paired sectors $F_i$ and $F_{i+1}$ in the given good pairing
  satisfies $r_{i+1}=F_i\cap F_{i+1}$, 
  and $F_i\neq F_j$ holds unless $\{i,j\}=\{0,k\}$. 
  As one goes from $F_0$ to $F_k$,
  if there exists one radius which is passed twice,
  then $\bigcup_{0\leq i\leq k}F_i$ is a folded disk.
  Otherwise $\bigcup_{0\leq i\leq k}F_i$ is a disk.
  By construction,
  these disks and folded disks intersect only at radii.
  By \Cref{def:goodPairing},
  no pair of these disks intersect properly.
  Finally, if $p$ is non-vertex non-manifold,
  \Cref{def:foldedDisk} dictates that ${\cal N}(p)$ 
  cannot contain any folded disks.
\end{proof}

\Cref{lem:diskDecomposition} guarantees the existence of 

\begin{definition}[Good disk decomposition]
  \label{def:goodDiskDecomp}
  For a non-manifold point $p\in \partial {\cal Y}$, 
  a \emph{good disk decomposition}
  of $\partial {\cal Y}\cap \overline{\mathcal{N}(p)}$ 
  is a finite set $\mathbf{D}(p):=\{{\cal D}_i: i=1,2,\ldots,m\}$
  of disks and/or folded disks such that
  $\partial {\cal Y}\cap \overline{\mathcal{N}(p)}$ coincides with
  \begin{equation}
    \label{eq:goodDiskDecomp}
    {\cal U}(\mathbf{D}(p)):= \cup_{{\cal D}_i\in \mathbf{D}(p)} {\cal D}_i
  \end{equation}
  where no pairs of disks in $\mathbf{D}(p)$ intersect properly
  and the improper intersection of any two distinct disks
  is either $p$ or a finite number of radii in ${\cal R}(p)$. 
\end{definition}

\begin{definition}[Locally bound]
  \label{def:goodDiskDecompBound}
  A good disk decomposition $\mathbf{D}(p)$ 
  of $\partial {\cal Y}\cap \overline{\mathcal{N}(p)}$
  is said to \emph{locally bound} ${\cal Y}$ or ${\cal Y}^{\perp}$
  if, for any $\mathcal{D}_i\in \mathbf{D}(p)$,
  each sector pair in $\mathcal{D}_i$ is in some good pairing
  within $\mathcal{N}(p)$
  that locally bounds ${\cal Y}$ or ${\cal Y}^{\perp}$, respectively. 
\end{definition}

The uniqueness of the good disk decomposition locally bounding
${\cal Y}$ is stated in

\begin{corollary}[Good disk decomposition of a good neighborhood]
  \label{coro:goodDiskDecomp}
  For a non-manifold point \mbox{$p\in \partial {\cal Y}$},
  the number $n_{\textrm{dd}}$ of different good disk decompositions of ${\cal N}(p)$
  is either 1 or 2.
  If $n_{\textrm{dd}}=1$, the good disk decomposition locally
  bounds both ${\cal Y}$ and ${\cal Y}^{\perp}$.
  If $n_{\textrm{dd}}=2$, 
  the good disk decomposition locally bounds ${\cal Y}$
  while the other locally bounds ${\cal Y}^{\perp}$.
\end{corollary}
\begin{proof}
  This follows from
  \Cref{lem:goodPairingsInNeighborhood}
  and \Cref{def:goodDiskDecomp,def:goodDiskDecompBound}. 
\end{proof}

For the Yin set in \Cref{fig:yinsetexm}(d,e),
 we have $n_{\textrm{dd}}=1$
 for the endpoints of both the horizontal and vertical yellow lines
 and $n_{\textrm{dd}}=2$
 for their intersection $q_X$. 
 


\section{Global topology} 
\label{sec:globalTopology}

To proceed from local topology
 to global topology in \Cref{thm:globalTopology},
 we need two operations on non-manifold points,
 one for non-vertex points in \Cref{def:unfoldingOp}
 and the other for isolated points in \Cref{def:detachingOp}.
 
By \Cref{def:VertexSingularPoint}, 
 the good neighborhood of a non-vertex non-manifold point $p$
 has two and only two radii $r_1$ and $r_2$.
A good disk decomposition $\mathbf{D}(p)$ has no folded disks 
 but disks, whose common intersection is $r=r_1\cup r_2$.

\begin{definition}[Unfolding]
  \label{def:unfoldingOp}
  For a non-vertex non-manifold point $p\in\partial {\cal Y}$, 
  the \emph{$\epsilon$-unfolding operation
    for a good disk decomposition} $\mathbf{D}(p)$
  of $\partial {\cal Y}\cap \overline{\mathcal{N}(p)}$ 
  is a continuous function
  $f: {\cal D}_{\sqcup}(p) \times [0,1]
  \rightarrow \overline{\mathcal{N}(p)}$
  where ${\cal D}_{\sqcup}(p):=
  \sqcup_{{\cal D}_i\in\mathbf{D}(p)} {\cal D}_i$ 
  with $\sqcup$ given in (\ref{eq:topologicalSum}) 
  and 
  \begin{enumerate}[(a)]
  \item $f({\cal D}_{\sqcup}(p) \times {0}) = {\cal U}(\mathbf{D}(p))$ 
    where ${\cal U}$ is given in (\ref{eq:goodDiskDecomp}); 
  \item for any $t \in (0,1]$,
    each ${\cal D}_i$ 
    and its image under $f|_{{\cal D}_i \times \{i\}\times \{t\}}$
    are homeomorphic; 
  \item $f({\cal D}_{\sqcup}(p) \times {1}) = {\cal U}(\tilde{\mathbf{D}}(p))$,
    where $\tilde{\mathbf{D}}(p) = \{\tilde{\cal D}_i\}$
    is a new set of disks such that
    (i) the restriction
    $f|_{\partial{\cal D}_i \times \{i\}\times\{t\}}$
    is an identity where $\partial{\cal D}_i$
    is defined in \Cref{def:genDisk}, 
    (ii) $\tilde{\cal D}_i\cap {\cal N}(p)$'s are pairwise disjoint,
    (iii) the disks in $\tilde{\mathbf{D}}(p)$
    have a one-to-one correspondence to those in $\mathbf{D}(p)$
    and, (iv) 
    the maximal distance of corresponding points
    in ${\cal U}(\mathbf{D}(p))$ and ${\cal U}(\tilde{\mathbf{D}}(p))$
    is no greater than $\epsilon$.
  \end{enumerate}
\end{definition}

The above unfolding operation is well defined
 since there always exists a function $f$
 that continuously deforms each disk ${\cal D}_i$
 to $\tilde{\cal D}_i$ in $\mathbb{R}^3$.
Furthermore, the unfolding can be performed one disk at a time
 so that no new intersections are introduced; 
 it can also be confined
 within an arbitrarily small neighborhood of the radius $r$
 so that the max-norm error of approximating
 each ${\cal D}_i$ with $\tilde{\cal D}_i$
 is no greater than $\epsilon$.
Consequently, 
 the unfolding operation eliminates 
 all non-vertex non-manifold points in ${\cal N}(p)$, 
 except those on $\partial{\cal N}(p)$. 
To understand (c)(i), 
 consider a vertex non-manifold point $p_v$
 of the Yin set ${\cal Y}$ in \Cref{fig:yinsetexm}(b)
 and a good neighborhood ${\cal N}(q)$
 of a non-vertex non-manifold point $q$
 close to $p_v$.
For ${\cal N}(q)$ satisfying $p_v\in\partial{\cal N}(q)$, 
 we \emph{cannot} separate $p_v$ into multiple points
 because $p_v$ is the \emph{single} center of a folded disk. 
Indeed, 
 a folded disk in ${\cal N}(q)$ should be transformed to a disk.

  
\begin{definition}[Detaching]
   \label{def:detachingOp}
  For an isolated non-manifold point $p\in\partial {\cal Y}$, 
  the \emph{$\epsilon$-detaching operation for a good disk decomposition} $\mathbf{D}(p)$
  of $\partial {\cal Y}\cap \overline{\mathcal{N}(p)}$ 
  is a continuous map
  $f: {\cal D}_{\sqcup}(p) \times [0,1] \rightarrow \overline{\mathcal{N}(p)}$
  satisfying (a,b,c) in \Cref{def:unfoldingOp}. 
\end{definition}

Different from those in \Cref{def:unfoldingOp},
 $\partial {\cal D}_i$'s of $\mathbf{D}(p)$ in \Cref{def:detachingOp}
 are pairwise disjoint
 and thus condition (c)(ii) implies that
 the disks in $\tilde{\mathbf{D}}(p)$ are pairwise disjoint.
In addition,
 the sequential perturbation of detaching one disk at a time
 is confined within an arbitrarily small neighborhood of $p$
 since there are no radii in ${\cal N}(p)$.

\begin{theorem}[Global topology of 3D Yin sets]
  \label{thm:globalTopology}
  The boundary of any Yin set ${\cal Y}\ne \emptyset, \mathbb{R}^3$
  is homeomorphic
  to the gluing of a finite collection of surfaces 
  along a 1D CW complex. 
\end{theorem}
\begin{proof}
  By \Cref{lem:CW-complex}(b),
  the non-manifold points of $\partial {\cal Y}$
  form a 1D CW complex. 
  Consider a maximally connected component $l$ of this 1D CW complex.
  If $l$ only consists of a single isolated non-manifold point,
  the conclusion holds trivially.
  Otherwise denote by $V$
  the set of all vertex non-manifold points in $l$.
  It is straightforward to select a set $P=\{p_i\in l: i=1,2,\ldots,m\}$
  of non-vertex non-manifold points
  and a good neighborhood for each point in $P$
  such that
  \begin{enumerate}[(a)]
  \item $i\ne j$ $\implies$ $p_{i}\not\in \mathcal{N}(p_{j})$,
    i.e., $p_{i}$ is contained only in $\mathcal{N}(p_{i})$; 
  \item $v\in V$ $\implies$ $v\not\in \cup_{i=1}^m\mathcal{N}(p_i)$; 
  \item for each radius $r$ at $v\in V$,
    $v\in \partial {\cal N}(p_i)$ and 
    $(r\setminus v) \subset {\cal N}(p_i)$ hold
    for some $p_i\in P$; 
  \item $l\subset \cup_{i=1}^m\overline{\mathcal{N}(p_i)}$.
  \end{enumerate}
  
  (b) states that no vertex non-manifold point is contained
  in any of the good neighborhoods
  while (c) states that each non-vertex non-manifold point
  in ${\cal N}(v)$ is contained in some ${\cal N}(p_i)$; 
  together they imply $(l\setminus V)\subset \cup_{i=1}^m\mathcal{N}(p_i)$.
  See \Cref{fig:nonvertexNeighborhoods}
  for an example of selecting $P$ and good neighborhoods
  of points in $P$. 
  By \Cref{coro:goodDiskDecomp},
  for each ${\cal N}(p_i)$ and each ${\cal N}(v)$,
  the good disk decomposition that locally bounds ${\cal Y}$
  is unique.
  By (d) and \Cref{lem:one2one},
  any two adjacent disk decompositions,
  either non-vertex/non-vertex or vertex/non-vertex, 
  are compatible.
  
  By \Cref{def:unfoldingOp},
  the application of the unfolding operation
  to each good disk decomposition in ${\cal N}(p_i)$ eliminates
  all non-vertex non-manifold point of ${\cal Y}$.
  It also converts each vertex non-manifold point of $l$
  to either an isolated non-manifold point or a manifold point;
  an example of the latter case is some $v\in V$
  where the good disk decomposition of ${\cal N}(v)$ 
  only consists of folded disks. 
  Now that all non-manifold points of $\partial \mathcal{Y}$
  are isolated,
  the application of the detaching operation in \Cref{def:detachingOp}
  to the good disk decomposition in the good neighborhood
  of each isolated non-manifold point
  yields a surface in \Cref{def:surface}. 

  Apply the above process to each connected component
  of non-manifold points of $\partial {\cal Y}$
  and we obtain a finite number of surfaces.
  Denote by $ \partial \tilde{\cal Y}$ the resulting space, 
  where each point $p\in \partial \tilde{\cal Y}$
  has a neighborhood homeomorphic to an open disk in $\mathbb{R}^2$.
  Since $\partial {\cal Y}$ is compact,
  $ \partial \tilde{\cal Y}$ is also compact.
  Therefore $ \partial \tilde{\cal Y}$
  is the union of a finite number of disjoint surfaces.
  Finally, the unfolding and detaching operations can be undone
  by first gluing at points in $V$
  and then gluing along the edges in the 1D CW complex.
\end{proof}

\begin{figure}
  \centering 
  \includegraphics[width=0.43\textwidth]{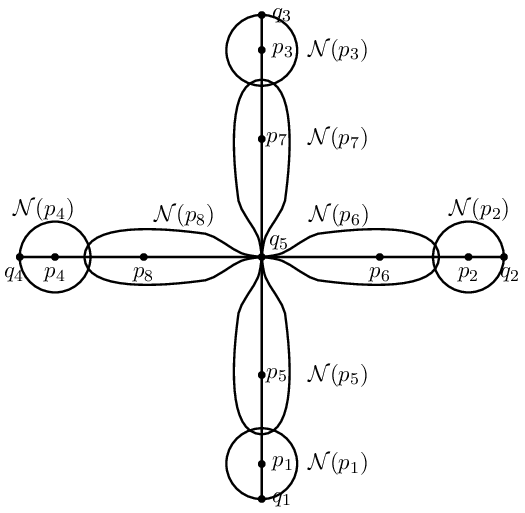}
  \caption{Illustrating the proof of \Cref{thm:globalTopology}
    with the Yin set ${\cal Y}$ in \Cref{fig:yinsetexm}(d,e).
    Each point in $\{q_i: i=1,2,3,4,5\}$
    is a vertex non-manifold point of ${\cal Y}$
    while each point in $\{p_i: i=1,\ldots,8\}$
    is a selected non-vertex non-manifold point.
    For each $i=1,2,3,4$,
    we have $q_i\in \partial{\cal N}(p_i)$ and 
    $q_i\not\in {\cal N}(p_i)$;
    the fact of ${\cal N}(p_i)\cap {\cal N}(p_{i+4})\ne \emptyset$
    implies the compatibility of their good disk decompositions.
    For each $j=5,6,7,8$,
    we have $q_5\in \partial N(p_j)$ and 
    $q_i\not\in {\cal N}(p_j)$;
    the fact of ${\cal N}(p_j)\cap {\cal N}(q_{5})\ne \emptyset$
    implies the compatibility of the good disk decomposition
    in ${\cal N}(q_{5})$ with that in each ${\cal N}(p_j)$.
  }
  \label{fig:nonvertexNeighborhoods}
\end{figure}

\section{Boundary representation}
\label{sec:representation}

\Cref{thm:globalTopology} leads to
 a unique B-rep of Yin sets
 in \Cref{thm:connectedYinsetrep}, 
 which states that the boundary of a 3D Yin set 
 can be uniquely decomposed
 into a collection of basic building blocks called glued surfaces. 
  

\begin{definition}[Glued surface]
  \label{def:gluedsurface}
  A \emph{glued surface} is a quotient space of a surface 
  such that
  (i) the quotient map glues the surface along 
  a 1D CW complex, 
  and (ii)
  the complement of the glued surface in $\mathbb{R}^3$
  has exactly two connected components,
  one bounded and the other unbounded.
\end{definition}

Condition (i) identifies the orientation of a glued surface 
 with that of the original surface
 while condition (ii) preserves the key property
 that the original surface has 
 a bounded complement and an unbounded complement in $\mathbb{R}^3$, 
 which is crucial in the unique boundary representation of Yin sets.
Then 
\Cref{def:intersectionsOfGluedSurfs,def:orientationOfSurface,def:interiorOfSurface} generalize
 to glued surfaces in a straightforward manner. 

\begin{definition}[Almost disjoint glued surfaces]
\label{def:almostDisjointGsurf}
 Two glued surfaces are \emph{almost disjoint}
  if they have no proper intersections
  and their improper intersections,
  cf. \Cref{def:intersectionsOfGluedSurfs},
  only consist of a finite number of isolated points and disjoint curves.
\end{definition}

\begin{definition}[Orientations of a glued surface]
  \label{def:orientationOfGluedSurface}
  A glued surface is \emph{positively oriented}
  if the normal vector $\mathbf{n}$ of the original surface
  always points from its bounded complement 
  to its unbounded complement;
  otherwise it is \emph{negatively oriented}.
\end{definition}


\begin{definition}[Internal complement of a glued surface]
  \label{def:interiorOfGluedSurface}
  The \emph{internal complement of a glued surface} $\mathcal{S}$,
  written $\mathrm{int}(\mathcal{S})$,
  is its bounded complement
  if $\mathcal{S}$ is positively oriented; 
  otherwise it is the unbounded complement.
\end{definition}

The \emph{external complement of a glued surface} $\mathcal{S}$
is the connected component of $\mathbb{R}^3\setminus{\cal S}$
that is not $\mathrm{int}(\mathcal{S})$.

\begin{definition}[Inclusion of glued surfaces]
  \label{def:inclusion}
  A glued surface $\mathcal{S}_k$ is said to \emph{include} $\mathcal{S}_l$,
  written $\mathcal{S}_k \ge \mathcal{S}_l$ or
  $\mathcal{S}_l \le \mathcal{S}_k$, if and only if
  the bounded complement of $\mathcal{S}_l$ is a subset
  of that of $\mathcal{S}_k$.
  If $\mathcal{S}_k$ includes $\mathcal{S}_l$ and
  $\mathcal{S}_k \neq \mathcal{S}_l$,
  we write $\mathcal{S}_k > \mathcal{S}_l$ or
  $\mathcal{S}_l < \mathcal{S}_k$.
\end{definition}

Equipped with the above partial ordering, 
 any collection of glued surfaces
 can be considered as
 a partially ordered set (poset) of glued surfaces. 

\begin{definition}[Covering of glued surfaces]
  \label{def:covering}
  Let $\mathcal{G}$ denote a poset of glued surfaces with inclusion
  as the partial order.
  We say $\mathcal{S}_k$ \emph{covers} $\mathcal{S}_l$ in $\mathcal{G}$
  and write "$\mathcal{S}_k \succ \mathcal{S}_{l}$" or
  "$\mathcal{S}_l \prec \mathcal{S}_k$" if
  $\mathcal{S}_l < \mathcal{S}_{k}$ and no elements
  $\mathcal{S} \in \mathcal{G}$ satisfy
  $\mathcal{S}_l < \mathcal{S} < \mathcal{S}_k$. 
\end{definition}


\begin{definition}[Glued-surface decomposition]
  \label{def:gluedSurfDecomp}
  A \emph{glued-surface decomposition
    of the boundary of a connected Yin set} ${\cal Y}\ne \emptyset, \mathbb{R}^3$
  is a poset ${\cal G}_{\partial {\cal Y}}=\{{\cal S}_j\}$
  of pairwise almost disjoint glued surfaces such that
  $\partial{\cal Y}
  = \cup_{{\cal S}_j\in {\cal G}_{\partial{\cal Y}}} {\cal S}_j$.
\end{definition}

By \Cref{lem:CW-complex}(b),
the set $Q$ of all non-manifold points of $\partial {\cal Y}$
form a 1D CW complex homeomorphic to a graph $G=(V,E,\phi_G)$, 
where each vertex $v\in V$ is either an isolated non-manifold point
or a vertex non-manifold point of $\partial {\cal Y}$
and $E$ contains the curves as the maximally connected subsets
of $Q\setminus V$.

\begin{definition}[Algorithm of glued-surface decomposition]
  \label{def:gluedSurfAlgorithm}
  Given a connected Yin set ${\cal Y}\ne \emptyset, \mathbb{R}^3$
  and the graph $G=(V,E,\phi_G)$ homeomorphic to the 1D CW complex
  of non-manifold points on $\partial{\cal Y}$, 
  the \emph{algorithm of glued-surface decomposition}
  is the following sequence of steps
  that construct a finite collection ${\cal G}_{\partial {\cal Y}}$
  of 2D subsets of $\mathbb{R}^3$.
  
  Firstly, we decompose $\partial {\cal Y}$
   to a set ${\cal G}_{cc}$ of surfaces and surfaces with boundary.
  \begin{enumerate}[(CC.1)]
  \item For each isolated non-manifold point $p\in V$, 
    apply the detaching operation in \Cref{def:detachingOp}
    to the good disk decomposition in $\overline{{\cal N}(p)}$. 
  \item For each curve $\gamma\in E$ with $m>2$ sectors on it,
    cut $\partial {\cal Y}$ open along $\gamma$, 
    i.e., replace $\gamma$ with $m$ pairwise disjoint curves
    $\gamma_1, \gamma_2, \ldots, \gamma_m$
    so that, for any $i\ne j$,
    the $i$th and the $j$th sectors on $\gamma$
    are now on disjoint radii $\gamma_i$ and $\gamma_j$. 
  \item Each maximally connected component of $\partial \tilde{\cal Y}$
    that results from (CC.2)
    is either a surface or a surface with boundary, 
    because all non-manifold points have been removed.
    Hereafter we denoted by ${\cal G}_{cc}$ the set of these
    components. 
  \end{enumerate}

  Second, we initialize ${\cal G}_{\partial {\cal Y}}$
  with all glued surfaces in ${\cal G}_{cc}$, 
  remove them from ${\cal G}_{cc}$, 
  and write ${\cal G}_{cc}=\{P_j\}$
  where each $P_j$ is a surface with boundary.
  
  Lastly, we finish the construction of ${\cal G}_{\partial {\cal Y}}$
  by steps as follows. 
  \begin{enumerate}[(GS.1)]
  \item Form a connected 2D subset $P_{GS}$ (with empty boundary)
    from ${\cal G}_{cc}$: 
    \begin{enumerate}[(a)]
    \item remove an arbitrary $P_j\in{\cal G}_{cc}$
      from ${\cal G}_{cc}$ as the starting point; 
    \item locate a boundary curve 
      $\gamma\subset \partial P_j$ and 
      find the unique $P_i\in{\cal G}_{cc}$
      satisfying that $(P_i,P_j)$ is a good pair on $\gamma$
      and locally bounds ${\cal Y}$; 
    \item glue $P_i$ to $P_j$ along $\gamma$, 
      remove $P_i$ from ${\cal G}_{cc}$,
      and update $P_j$ with $P_i\cup P_j$;
    \item repeat (b,c) until the boundary of $P_j$ becomes empty
      and set $P_{GS}=P_j$. 
    \end{enumerate}
  \item Decompose $P_{GS}$ into a finite collection
    $\mathbf{S}_{GS}=\{{\cal S}_j\}$ of 2D subsets of $\mathbb{R}^3$:
    \begin{enumerate}[(a)]
    \item cut $P_{GS}$ open along curves in $E$
      in the same way
      as that of (CC.2) to obtain a collection of surfaces with boundary;
    \item glue these surfaces with boundary
      to form $\mathbf{S}_{GS}$, 
      a collection of 2D subsets with empty boundary,
      by the same substeps in (GS.1a--d)
      except that here a good pairing $(P_i,P_j)$ on $\gamma$
      is required to locally bound ${\cal Y}^{\perp}$ instead of
      ${\cal Y}$; 
    \end{enumerate}
  \item Update ${\cal G}_{\partial {\cal Y}}$
    with ${\cal G}_{\partial {\cal Y}}\cup \mathbf{S}_{GS}$; 
  \item Repeat (GS.1--3) until ${\cal G}_{cc}$ is empty.
  \end{enumerate}
\end{definition}

\begin{figure}
  \label{fig:uniquerep}
  \centering
  \subfigure[The Yin set ${\cal Y}$]{
    \includegraphics[width=0.45\textwidth]{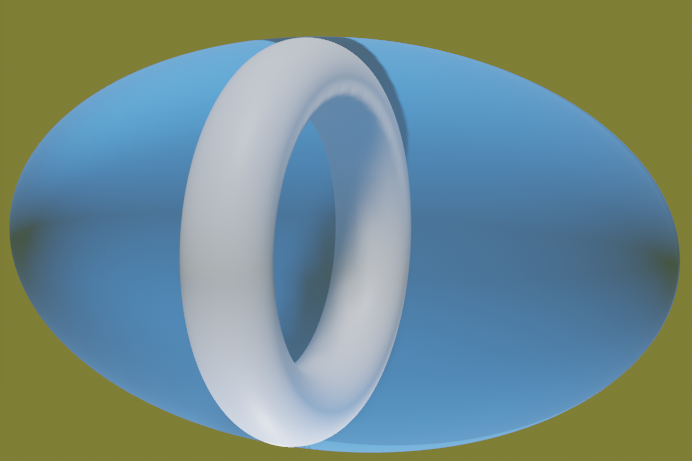}
  }
  \hfill  
  \subfigure[A spanwise cross section of $P_{GS}$]{
    \includegraphics[width=0.45\textwidth]{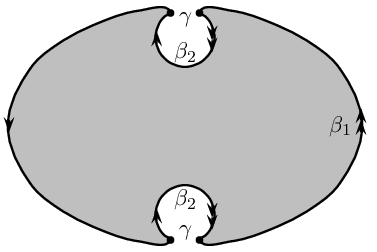}
  }
  
  \subfigure[Cross sections of ${\cal S}_j$'s in $\mathbf{S}_{GS}$]{
    \includegraphics[width=0.35\textwidth]{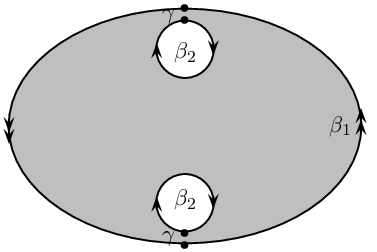}
  }
  \hfill
  \subfigure[${\cal G}_{\partial{\cal Y}}$ constructed
  by \Cref{def:gluedSurfAlgorithm}]{
    \includegraphics[width=0.55\textwidth]{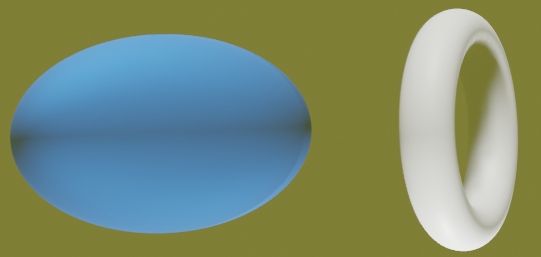}
  }
  \caption{Illustrating the algorithm of glued-surface decomposition 
    in \Cref{def:gluedSurfAlgorithm}.
    The Yin set ${\cal Y}$ is obtained
    by removing a solid torus from an elliptical ball
    so that all non-manifold points on $\partial {\cal Y}$
    form a Jordan curve,
    which is also the only (improper) intersection
    of the ellipsoid and the torus.
    After the execution of (CC.1--3),
    ${\cal G}_{cc}$ contains three surfaces with boundary $P_1$, $P_2$, $P_c$, 
    where $P_1$ and $P_2$ are homeomorphic to the unit half sphere
    and $P_c$ to the cylinder $\mathbb{S}^1\times[0,1]$.
    Before (GS.1), 
    ${\cal G}_{\partial {\cal Y}}$ is still empty
    because (CC.1--3) yield no glued surfaces.
    If the half sphere $P_1$ is chosen
    as the starting point in (GS.1a),
    then $P_{GS}$ is formed as
    $P_1\leftrightarrow P_c \leftrightarrow P_2$
    where $\leftrightarrow$ denotes the gluing.
    As shown in subplot (b),
    $\mathbb{R}^3\setminus P_{GS}$ has three connected components,
    two bounded and one unbounded. 
    In (GS.2), $P_{GS}$ is decomposed
    into two 2D subsets with empty boundary,
    whose cross sections are shown in subplot (c).
    They turn out to be the two surfaces shown in subplot (d).
  }
\end{figure}

We emphasize that the steps (GS.2--4) are necessary
 because $P_{GS}$ needs not to be a glued surface; 
 see the example shown 
 in \Cref{fig:uniquerep} to illustrate \Cref{def:gluedSurfAlgorithm}.
 
\begin{lemma}[Unique glued-surface decomposition of 3D Yin sets]
  \label{lem:gluedSurfDecomp}
  For any connected Yin set
  ${\cal Y}\ne \emptyset,\mathbb{R}^3$, 
  the output ${\cal G}_{\partial {\cal Y}}$
  of the algorithm in \Cref{def:gluedSurfAlgorithm}
  is the unique glued-surface decomposition of
  $\partial{\cal Y}$ as in \Cref{def:gluedSurfDecomp}.
\end{lemma}
\begin{proof}
  We need to prove 
  (i) the construction of ${\cal G}_{\partial {\cal Y}}$ is unique, 
  (ii) each ${\cal S}_j\in {\cal G}_{\partial {\cal Y}}$
  is a glued surface,
  and (iii) the glued surfaces in ${\cal G}_{\partial {\cal Y}}$
  are pairwise almost disjoint.

  (i) follows from the uniqueness of the algorithmic steps
  in \Cref{def:gluedSurfAlgorithm}.
  In particular,
  the uniqueness of (GS.1b) and that of (GS.2b)
  follow from \Cref{lem:goodPairingsRadius}.

  To prove (ii), we start by observing that
  each element in ${\cal G}_{cc}$ is either a surface
  or a surface with boundary.
  Although $P_{GS}$ produced by (GS.1) has no boundary, 
  it may violate condition (ii) in \Cref{def:gluedsurface}
  and thus fail to be a glued surface.
  Nonetheless,
  by ${\cal S}_j\subset P_{GS}\subset \partial {\cal Y}$,
  the condition of locally bounding ${\cal Y}$ in (GS.1b),
  and the connectedness of ${\cal Y}$,
  one and \emph{only} one connected component $C_{\cal Y}$
  of $\mathbb{R}^3\setminus P_{GS}$ contains ${\cal Y}$.
  By \Cref{lem:localTopoOfYinSets}(c),
  $\partial{\cal Y}\setminus {\cal S}_j$
  must also be contained in $C_{\cal Y}$. 
  On the other hand, 
  ${\cal Y}^{\perp}$ cannot be distributed
  in multiple connected components
  of $\mathbb{R}^3\setminus {\cal S}_j$
  other than $C_{\cal Y}$,
  because that would contradict
  the condition of locally bounding ${\cal Y}^{\perp}$ in (GS.2b).
  For any point $p\not\in {\cal S}_j$,
  one and only one of $p\in {\cal Y}$, $p\in {\cal Y}^{\perp}$,
  and $p\in \partial{\cal Y}\setminus {\cal S}_j$ holds.
  Therefore, 
  $\mathbb{R}^3\setminus {\cal S}_j$
  only consists of two components,
  one containing ${\cal Y}$ and the other being contained in ${\cal Y}^{\perp}$.
  Thus by \Cref{def:gluedsurface} 
  ${\cal S}_j$ is indeed a glued surface. 

  (iii) follows from \Cref{def:almostDisjointGsurf}
  and the fact of all improper intersections of the glued surfaces
  being preserved,
  i.e., each detaching or cutting-open operation
  is matched with a gluing operation.
\end{proof}

The main result of this section is
\begin{theorem}[Unique B-rep of connected Yin sets]
  \label{thm:connectedYinsetrep}
  The boundary $\partial {\cal Y}$
  of any connected Yin set
  ${\cal Y}\ne \emptyset,\mathbb{R}^3$ 
  can be uniquely decomposed into a finite poset
  \mbox{$\mathcal{G}_{\partial \mathcal{Y}}=\{\mathcal{S}_j\subset \partial {\cal Y}\}$}
  of pairwise almost disjoint oriented glued surfaces
  such that 
  \begin{equation}
    \label{equ:connectedYinsetrep}
  \mathcal{Y} =
  \mathop{\bigcap}\limits_{\mathcal{S}_{j}\in \mathcal{G}_{\partial \mathcal{Y}}}
  \mathrm{int}(\mathcal{S}_{j}).
\end{equation}
Furthermore, the poset $\mathcal{G}_{\partial \mathcal{Y}}$
must be one of the following two types,
\begin{equation}
  \label{equ:setgsurfacetype}
   \left\{
    \begin{array}{lr}  
      \mathcal{G}^{-} = \left\{ \mathcal{S}^-_1,
    \mathcal{S}_2^-, \cdots, \mathcal{S}^-_{n_-} \right\},
    \quad &n_- \ge 1, \\
\mathcal{G}^{+} = \left\{ \mathcal{S}^+, \mathcal{S}^-_1,
    \mathcal{S}_2^-, \cdots, \mathcal{S}^-_{n_-}\right\} ,
    \quad &n_- \ge 0,  \\
    \end{array}
    \right.
\end{equation}
where all $\mathcal{S}_j^-$'s are negatively oriented
and mutually incomparable
with respect to inclusion.
For $\mathcal{G}^{+}$, we also have 
\begin{equation}
  \label{equ:gsurfacerelation}
\forall j = 1,2,\cdots,n_-,\quad \mathcal{S}_j^- \prec \mathcal{S}^+.
\end{equation}
\end{theorem}
\begin{proof}
  By \Cref{lem:gluedSurfDecomp},
  the poset $\mathcal{G}_{\partial \mathcal{Y}}$ 
  is the unique decomposition of $\partial \mathcal{Y}$.
  For each ${\cal S}_j\in \mathcal{G}_{\partial \mathcal{Y}}$, 
   we orient it according to \Cref{def:orientationOfGluedSurface,def:interiorOfGluedSurface}
   so that the internal complement of ${\cal S}_j$ always contains ${\cal Y}$.
  For example, in \Cref{fig:uniquerep}(d)
   the ellipsoid is positively oriented while
   the torus is negatively oriented. 
   
  The connectedness of ${\cal Y}$ implies
   that ${\cal Y}$ is contained
   in a single complement of each ${\cal S}_j$. 
  By \Cref{def:gluedsurface},
   $\mathbb{R}^3\setminus{\cal S}_j$ only has two connected components
   and we have
   ${\cal Y}\subset \Int({\cal S}_j)$
   for each ${\cal S}_j\in \mathcal{G}_{\partial \mathcal{Y}}$.
  Then $\partial{\cal Y}=\cup_j {\cal S}_j$ gives
  (\ref{equ:connectedYinsetrep}). 
   
  Suppose $\mathcal{G}_{\partial \mathcal{Y}}$
  has two or more positively oriented glued surfaces ${\cal S}_1$ and ${\cal S}_2$.
  Then \Cref{def:interiorOfGluedSurface} implies
  that both bounded complements of ${\cal S}_1$ and ${\cal S}_2$
  are subsets of ${\cal Y}$,
  which contradicts ${\cal Y}$ being entirely contained
  in only one complement of each ${\cal S}_j$. 
  
  If $\mathcal{G}_{\partial \mathcal{Y}}$ has no positively oriented
  glued surfaces, 
  $\mathcal{Y} \neq \emptyset$ dictates that 
  at least one negatively oriented glued surface
  exists in $\mathcal{G}_{\partial \mathcal{Y}}$
  and thus $\mathcal{G}_{\partial \mathcal{Y}}$
  must be of the $\mathcal{G}^{-}$ type.
  If $n_- = 1$, (\ref{equ:setgsurfacetype}) holds trivially.
  For $n_->1$, any two glued surfaces $\mathcal{S}_1^-$ and $\mathcal{S}_2^{-}$
  must be incomparable;
  otherwise it would contradict the connectedness of $\mathcal{Y}$.

  If $\mathcal{G}_{\partial \mathcal{Y}}$ has
  one positively oriented glued surfaces, 
  then $\mathcal{G}_{\partial \mathcal{Y}}$ must
  be of the $\mathcal{G}^{+}$ type.
  If $n_- = 0$, (\ref{equ:setgsurfacetype})
  and (\ref{equ:gsurfacerelation}) hold trivially.
  For $n_->0$, 
  suppose (\ref{equ:gsurfacerelation}) did not hold
  for some negatively oriented glued surface $\mathcal{S}_1^-$.
  Then $\mathcal{S}_1^- \succ \mathcal{S}^+$ or
  they are incomparable.
  For the former case, 
   a path from one point in the internal complement of $\mathcal{S}_1^-$
   to another point in the internal complement of $\mathcal{S}^+$ must
   contain some points not in $\mathcal{Y}$,
   and this contradicts the connectedness of $\mathcal{Y}$.
  The latter case also contradicts the connectedness of $\mathcal{Y}$.
  Thus  (\ref{equ:gsurfacerelation}) holds.
  In addition, arguments similar to those in the previous paragraph
  imply that the negatively oriented glued surfaces are pairwise incomparable.
\end{proof}

\Cref{thm:connectedYinsetrep} is applied to
 each connected component of a general Yin set to give

\begin{corollary}[Unique B-rep of Yin sets]
  \label{coro:Yinsetrep}
  For any Yin set $\mathcal{Y} \neq \emptyset, \mathbb{R}^{3}$,
  its boundary $\partial {\cal Y}$ can be uniquely decomposed into a
  collection
  \begin{equation}
    \label{eq:BrepSetsOfYinSet}
    {\cal G}_{\partial {\cal Y}}=\bigl\{
    {\cal G}_i=\{{\cal S}_{i,j}\subset \partial {\cal Y}_i\}:
  {\cal Y}_i \text{ is a connected component of } {\cal Y}
  \bigr\}
  \end{equation}
  of posets of pairwise almost disjoint oriented glued surfaces
  such that 
  \begin{equation}
    \label{equ:Yinsetrep}
    \mathcal{Y} = \bigcup\nolimits^{\perp\perp}_{i}
    \mathop{\bigcap}\nolimits_{j}
    \mathrm{int}(\mathcal{S}_{i,j}), 
  \end{equation}
  where each ${\cal G}_i$ satisfies \Cref{thm:connectedYinsetrep}. 
\end{corollary} 

\Cref{coro:Yinsetrep} furnishes
 an appealing advantage of the B-rep
 that global topological invariants of any 3D Yin set ${\cal Y}$
 can be extracted in $O(1)$ time. 
In particular, the number of connected components of ${\cal Y}$
 is 1 if ${\cal G}_{\partial {\cal Y}}$ is of the type ${\cal G}^-$;
 otherwise it is the number of positively oriented glued surfaces
 in ${\cal G}_{\partial {\cal Y}}$. 
Also, the number of holes (i.e., 2-cycles)
 in the closure $\overline{{\cal Y}_i}$ of a connected component of ${\cal Y}$
 is the number of negatively oriented glued surfaces
 in ${\cal G}_{\partial {\cal Y}_i}$. 




\section{Conclusion}
\label{sec:conclusion}

We have proposed the Yin space
 as a model of 3D continua with arbitrarily complex topology.
In particular,
 we have characterized the local and global topology
 of singular points on the boundary of a 3D Yin set.
The main theorem on the global topology of 3D Yin sets
 also leads to their unique B-rep via glued surfaces.
This work is our first step of 3D fluid modeling
 in the study of multiphase flows.

Several future research prospects follow.
By \Cref{thm:globalTopology,thm:connectedYinsetrep}
 and the fact that all compact 2-manifolds can be triangulated, 
 it suffices to employ a simplicial 2-complex
 as the underlying data structure
 for the B-rep and 3D IT of any 3D continua. 
Such a linear MARS for 3D IT is already done \cite{qiu25:_mars}
 and the development of fourth- and higher-order IT methods
 is currently a work in progress.
Other interesting topics of IT for deforming 3D continua
 are the theoretical characterizations
 and algorithmic treatments of topological changes.
Finally, to integrate CAD to CAE,
 we will equip the 3D Yin space with a Boolean algebra
 so that the discretization of spatial operators
 in numerically solving partial differential equations
 can benefit from the explicit modeling
 of topology and geometry.


{\bf Acknowledgments.}
We acknowledge helpful comments
 from Shaozhen Cao, Junxiang Pan, and Chenhao Ye, 
 graduate students at the school of mathematical sciences
 in Zhejiang University.

\bibliographystyle{siamplain}
\bibliography{bib/YinSets3D}
\end{document}